\documentclass{amsart}
\usepackage{amsmath,amssymb,amsfonts}
\usepackage[mathscr]{eucal}

\usepackage{enumerate}

\theoremstyle{plain}
\newtheorem{theorem}{Theorem}[section]
\newtheorem{cor}[theorem]{Corollary}
\newtheorem{prop}[theorem]{Proposition}
\newtheorem{lemma}[theorem]{Lemma}

\theoremstyle{definition}
\newtheorem{definition}[theorem]{Definition}
\newtheorem{example}[theorem]{Example}
\newtheorem{remark}[theorem]{Remark}

\input xy
\xyoption{all}

\newcommand{\Deltaop}{\Delta^{op}}
\newcommand{\Thetanop}{\Theta_n^{op}}
\newcommand{\SSets}{s\mathcal{S}ets}

\newcommand{\Map}{\operatorname{Map}}
\newcommand{\map}{\operatorname{map}}
\newcommand{\Mapobj}{\operatorname{\underbar{Map}}}
\newcommand{\mapobj}{\operatorname{\underbar{map}}}
\newcommand{\Hom}{\operatorname{Hom}}
\newcommand{\colim}{\operatorname{colim}}
\newcommand{\hocolim}{\operatorname{hocolim}}

\newcommand{\Sets}{\mathcal{S}ets}
\newcommand{\Secat}{\mathcal Se \mathcal Cat}

\newcommand{\cosk}{\text{cosk}}
\newcommand{\css}{\mathcal{CSS}}
\newcommand{\Ho}{\operatorname{Ho}}
\newcommand{\ob}{\operatorname{ob}}

\newcommand{\Se}{\operatorname{Se}}
\newcommand{\Thetansp}{\Theta_nSp}
\newcommand{\id}{\operatorname{id}}
\newcommand{\diag}{\operatorname{diag}}
\newcommand{\heq}{\operatorname{heq}}
\newcommand{\sPsh}{s\mathcal{PS}h}
\newcommand{\Coll}{\operatorname{Coll}}
\newcommand{\Cpt}{\operatorname{Cpt}}
\newcommand{\Rec}{\operatorname{Rec}}
\newcommand{\acss}{\mathcal{ACSS}}

\def\noloc{\;{:}\,}

\begin{document}

\title[$(\infty,n)$-categories]{Comparison of models for $(\infty, n)$-categories, II}

\author[J.E.\ Bergner]{Julia E.\ Bergner}

\address{Department of Mathematics, University of Virginia, Charlottesville, VA}

\email{bergnerj@member.ams.org}

\author[C.\ Rezk]{Charles Rezk}

\address{Department of Mathematics, University of Illinois at Urbana-Champaign, Urbana, IL}

\email{rezk@math.uiuc.edu}

\date{\today}

\subjclass[2010]{55U35, 55U40, 18D05, 18D15, 18D20, 18G30, 18G55}

\keywords{$(\infty, n)$-categories, $\Theta_n$-spaces, complete Segal objects}

\thanks{The first-named author was partially supported by NSF grants DMS-1105766 and DMS-1352298.  The second-named author was partially supported by NSF grant DMS-1006054.  Much of this work was done while both authors were in residence at MSRI in Spring 2014, supported by NSF grant 0932078 000.}

\begin{abstract}
In this paper we complete a chain of explicit Quillen equivalences between the model category for $\Theta_{n+1}$-spaces and the model category of small categories enriched in $\Theta_n$-spaces.  The Quillen equivalences given here connect Segal category objects in $\Theta_n$-spaces, complete Segal objects in $\Theta_n$-spaces, and $\Theta_{n+1}$-spaces.
\end{abstract}

\maketitle

\section{Introduction}

In \cite{rezktheta}, the second-named author developed the theory of $\Theta_n$-spaces as models for $(\infty, n)$-categories and showed that they are the fibrant objects in a cartesian model category.   Since one of the expected properties of a good model for $(\infty, n)$-categories is that categories enriched in them should be models for $(\infty, n+1)$-categories, a natural conjecture was that this property should hold for $\Theta_n$-spaces.

The first step in proving this conjecture was establishing a model structure on the category of small categories enriched in $\Theta_n$-spaces.  Because the model category for $\Theta_n$-spaces is sufficiently nice (with the cartesian property being particularly key), we were able to use a theorem of Lurie to obtain the desired model structure, as we showed in \cite{part1}.

While a direct Quillen equivalence between the two model categories was not expected, there was a conjectured chain of Quillen equivalences between intermediate model categories, mostly following by analogy from the case of $(\infty, 1)$-categories, where the simplicial category and complete Segal space model categories were shown to be equivalent \cite{thesis}.  In \cite{part1}, we established the first two Quillen equivalences in this chain, proving that categories enriched in $\Theta_n$-spaces are equivalent to  Segal category objects in $\Theta_n$-spaces.

In this paper, we complete this conjectured chain of Quillen equivalences.  We prove that there is a Quillen equivalence between the model structure for Segal category objects in $\Theta_n$-spaces and the model structure for complete Segal objects in $\Theta_n$-spaces.  We also establish one final Quillen equivalence between the model structure for complete Segal objects in $\Theta_n$-spaces and the model structure for $\Theta_{n+1}$-spaces.  This equivalence is not present in the case when $n=0$, since the two notions coincide precisely as complete Segal spaces.

This last Quillen equivalence is the first in a chain interpolating between the $\Theta_n$-space model and the $n$-simplicial model for $(\infty, n)$-categories as defined by Barwick and Lurie \cite{lurie2}, \cite{luriecob}.   Both models were already known to be equivalent on the level of $(\infty,1)$-categories using the axiomatic methods of Barwick and Schommer-Pries \cite{bsp}, but we give an alternative approach using an explicit Quillen equivalence on the level of model categories.  Another approach to some of the results in our work, again on the level of $(\infty, 1)$-categories, is given by Haugseng \cite{haugseng}, \cite{haugseng2}.  In addition to providing a comparison which employs the full strength of model categories, rather than their underlying $(\infty,1)$-categories, our proof relies on different tools than these other approaches.  In particular, we work primarily in the context of localized model structures for presheaf categories and do not require the framework of category theory in quasi-categories.

Furthermore, much of the present paper is worked in a substantially more general context than simply in $(\infty,n)$-categories.  The categories $\Theta_n$ are examples of categories formed by the $\Theta$-construction, and the model category framework for $\Theta_n$-spaces of \cite{rezktheta} is developed in this more general setting.  Here, we define complete Segal objects in this same degree of generality and make the comparison accordingly.  However, more general Segal category objects present considerable technical difficulties, so that our comparison here only holds between Segal category objects and complete Segal objects in $\Theta_n$-spaces.

Many other approaches to $(\infty, n)$-categories are being investigated; some references include \cite{ara}, \cite{bk}, \cite{hs}, \cite{lurie2}, \cite{luriecob}, \cite{pel}, \cite{simpson}, and \cite{verity}.  We give a more substantial survey in the first paper in this series \cite{part1}.

We give a brief outline of the contents of the paper.  In Section 2, we set up some notation and review basic facts about simplicial presheaf categories and cartesian structures.  Section 3 is concerned with the $\Theta$-construction and relations between various presheaf categories.  In Sections 4 and 5, we develop the model structures for $\Theta$-objects and complete Segal objects, respectively.  In Section 6, we give a Quillen equivalence between these model structures; Section 7 states the results in the special case of models for $(\infty, n)$-categories.  We consider various notions of equivalence of Segal space objects in Section 8, and these results are used in Section 9 to give the Quillen equivalence between Segal category objects and complete Segal objects.

\section{Notation and background}

\subsection{Simplicial objects and presheaf categories} \label{presheaf}

We write $\Delta$ for the simplicial indexing category, with objects
\[ [n]=\{0 \leq 1 \leq \cdots \leq n\} \] and morphisms weakly order-preserving maps.  We write $\delta^{k_0,\dots,k_m}$ for the map $[m] \rightarrow [n]$ sending $i \mapsto k_i$ for each $0 \leq i \leq m$.

We sometimes extend the notation for objects in the following way.  For $p \leq q$ we write $[p,q]=\{p \leq p+1 \leq \cdots \leq q\}$.  As an object of $\Delta$, it is identified with $[q-p]$; the notation is meant to indicate that it is equipped with a distinguished inclusion $[p,q] \rightarrow [m]$ for $m \geq q$, where for each $p \leq i \leq q$ we have $i \mapsto i$.

Let $\SSets$ denote the category of simplicial sets, or functors $\Deltaop \rightarrow \Sets$.   For a small category $\mathcal C$, let $\sPsh(\mathcal C)$ denote the category of presheaves of simplicial sets on $\mathcal C$, or functors $\mathcal C^{op} \rightarrow \SSets$.  We write $F=F_{\mathcal C} \colon \mathcal C \rightarrow \sPsh(\mathcal C)$ for the Yoneda embedding, defined by the representable object $F_\mathcal C(c) = \Hom(-, c)$ for any object $c$ of $\mathcal C$, where the right-hand side is regarded as a discrete simplicial set.

Given a functor $f \colon \mathcal C \rightarrow \mathcal D$, we have a sequence of adjoint functors $f_\#\dashv f^*\dashv f_*$ of the form
\[\xymatrix{ {\sPsh(\mathcal C)} \ar@<-6ex>[d]^{f_\#} \ar@<6ex>[d]^{f_*}
\\ {\sPsh(\mathcal D)} \ar[u]_{f^*} }\]
where $f^*$ is the restriction functor, defined by $(f^*X)(c)=X(fc)$ for any object $c$ of $\mathcal C$. We note that there is a natural isomorphism of functors
\[f_\# F_{\mathcal C} \xrightarrow{\cong} F_{D} \colon \mathcal C \rightarrow \sPsh(\mathcal D) \]
obtained using the adjointness of $f_\#$ and $f^*$.

We make use of the following observation: if $f\colon \mathcal C \rightleftarrows \mathcal D \noloc g$ is an adjoint pair, then there are natural isomorphisms $f^* \cong g_\#$ and $f_* \cong g^*$, and thus
an adjoint sequence of the form $f_\# \dashv f^* \cong g_\# \dashv f_* \cong g^*\dashv g_*$.

Let 1 denote the category with a single object and no non-identity morphisms, suppose that $\mathcal C$ is equipped with a terminal object $t_\mathcal C$, and consider the resulting adjoint pair $p \colon \mathcal C \rightleftarrows 1 \noloc t$, where $t(1)=t_\mathcal C$.  Because $p \circ t = \id$, the functor $p^*\cong t_\# \colon \SSets = \sPsh(1) \rightarrow \sPsh(\mathcal C)$, which sends a simplicial set $X$ to the constant functor with value $X$,
is fully faithful.  Thus $\SSets$ is equivalent to the full subcategory of constant functors in $\sPsh(\mathcal C)$, and we sometimes make this identification silently.

More generally, if $\mathcal C$ and $\mathcal D$ each have terminal objects, then we obtain full embeddings $\pi_\mathcal C^*\colon \sPsh(\mathcal C) \rightarrow \sPsh(\mathcal C \times \mathcal D)$ and $\pi_\mathcal D^*\colon \sPsh(\mathcal D) \rightarrow \sPsh(\mathcal C \times \mathcal D)$, induced by the two projection functors.

\setcounter{theorem}{1}

\begin{example}
We are most interested in the categories $\Delta$ and $\Theta_n$.  (We review the definition of $\Theta_n$ is the next section.)  Both of these categories have terminal objects, so $\sPsh(\Delta)$ and $\sPsh(\Theta_n)$ can be regarded as full subcategories of $\sPsh(\Delta \times \Theta_n)$.  In particular, we regard both Yoneda functors $F_\Delta$ and $F_{\Theta_n}$ as objects of $\sPsh(\Delta \times \Theta_n)$, and $F_{\Delta \times \Theta_n}([m], \theta) \cong F_\Delta([m]) \times F_{\Theta_n}$.  For simplicity, we write $F_\Theta$ for $F_{\Theta_n}$.
\end{example}

Unless indicated otherwise, when considering $\sPsh(\mathcal C)$, we assume it has the injective model structure, where weak equivalences and cofibrations are defined levelwise in $\SSets$.  We often localize this model structure with respect to a set of maps $\mathcal S$, and denote the resulting model structure by $\sPsh(\mathcal C)_{ \mathcal S}$.  An $\mathcal S$-\emph{fibrant object} is an injective fibrant object $X$ such that $\Map_{\sPsh(\mathcal C)}(f,X)$ is a weak equivalence for all $f \in \mathcal{S}$.  An $\mathcal{S}$-\emph{local object} is any object levelwise weakly equivalent to an $\mathcal{S}$-fibrant object.  We write $\overline{\mathcal{S}}$ for the class of weak equivalences in this model category.  The pair $(\mathcal C, \mathcal S)$ of a small category $\mathcal C$ together with a set $\mathcal S$ of maps in $\sPsh(\mathcal C)$ is sometimes called a \emph{presentation}.

Throughout this paper morphisms between two objects can take several different forms.  By $\Hom$ we denote morphism set, by $\Map$ or $\map$ we denote morphism simplicial set, and by $\Mapobj$ or $\mapobj$ we denote internal hom object in the relevant category.  (In some places we have additional notation to clarify the category in which we take this internal hom; this notation will be explained as needed.)

\setcounter{subsection}{2}

\subsection{Cartesian structures}

The additional structure of a cartesian model category is a key feature of some of the model categories which we consider in this paper.  We give a brief review of the definitions here, and the reader is referred to \cite[\S 2]{rezktheta} for a detailed treatment.

\setcounter{theorem}{3}

\begin{definition}
A category $\mathcal C$ is \emph{cartesian closed} if it has finite products and, for any two objects $X$ and $Y$ of $\mathcal C$, an internal function object $Y^X$, together with a natural isomorphism
\[ \Hom_\mathcal C(Z, Y^X) \cong \Hom_\mathcal C(Z \times X, Y) \]
for any third object $Z$ of $\mathcal C$.
\end{definition}

If a cartesian closed category additionally has a model structure, we can ask if these two structures are compatible, in the following sense.

\begin{definition}
A model category $\mathcal M$ is \emph{cartesian} if its underlying category is cartesian closed, the terminal object is cofibrant, and the following equivalent conditions hold.
\begin{enumerate}
\item If $f \colon A \rightarrow A'$ and $g \colon B \rightarrow B'$ are cofibrations in $\mathcal M$ , then the induced map
\[ h \colon A \times B' \coprod_{A \times B} A' \times B \rightarrow A' \times B' \]
is a cofibration.  If either $f$ or $g$ is a weak equivalence then so is $h$.

\item If $f \colon A \rightarrow A'$ is a cofibration and $p \colon X' \rightarrow X$ is a fibration in $\mathcal M$, then the induced map
\[ q \colon (X')^{A'} \rightarrow (X')^A \times_{X^A} X^{A'} \]
is a fibration.  If either $f$ or $p$ is a weak equivalence, then so is $q$.
\end{enumerate}
\end{definition}

We use the following proposition, which was proved in the special case of simplicial spaces in \cite[9.2]{rezk}.

\begin{prop} \label{cartesian}  \cite[2.10, 2.22]{rezktheta}
Let $\mathcal C$ be a small category and $F_\mathcal C \colon \mathcal C \rightarrow \sPsh(\mathcal C)$ the Yoneda embedding. Consider the category $\sPsh(\mathcal C)$ with the injective model structure and consider its localization with respect to a set $\mathcal S$ of maps.  Suppose that if $W$ is an $\mathcal S$-fibrant object, then so is $W^{F_\mathcal C(c)}$ for any object $c$ of $\mathcal C$.  Then the localized model structure is compatible with the cartesian closure.
\end{prop}

We can incorporate cartesian-ness with the notion of presentation.

\begin{definition}
A presentation $(\mathcal C, \mathcal S)$ is a \emph{cartesian presentation} if, for every $\mathcal S$-local object $X$ in $\sPsh(\mathcal C)$ and $s \colon S \rightarrow S'$ in $\mathcal S$, the induced map $Y^s \colon Y^{S'} \rightarrow Y^S$ is a levelwise weak equivalence, where $Y$ is a fibrant replacement of $X$.
\end{definition}

\begin{prop} \cite[2.11]{rezktheta}
Let $(\mathcal C, \mathcal S)$ be a presentation.  Then the following are equivalent.
\begin{enumerate}
\item The presentation $(\mathcal C, \mathcal S)$ is a cartesian presentation.

\item For any $\mathcal S$-fibrant $X$ in $\sPsh(\mathcal C)$ and any $c \in \ob(\mathcal C)$, the object $X^{Fc}$ is $\mathcal S$-local.

\item For any $s \colon S \rightarrow S'$ in $\mathcal S$ and any $c \in \ob(\mathcal C)$, the map $s \times Fc \colon S \times Fc \rightarrow S' \times Fc$ is in $\overline{\mathcal S}$.
\end{enumerate}
\end{prop}

\section{The $\Theta$-construction and comparison functors}

In this section we recall the $\Theta$-construction for obtaining a new category $\Theta \mathcal C$ from a category $\mathcal C$, and we set up a comparison between the categories $\Theta \mathcal C$ and $\Delta \times \mathcal C$.  Some of the material in this section can be found in more detail in \cite[\S 3-4]{rezktheta}.

Let $\mathcal C$ be a small category, and define the category $\Theta \mathcal C$ to have objects $[m](c_1, \ldots ,c_m)$ where $[m]$ is an object of $\Delta$ and each $c_i$ is an object of $\mathcal C$.  A morphism
\[ [m](c_1, \ldots, c_m) \rightarrow [p](d_1, \ldots, d_p)\]
is given by a morphism $\delta \colon [m] \rightarrow [p]$ in $\Delta$ together with a morphism $f_{ij} \colon c_i \rightarrow d_j$ in $\mathcal C$ for every $1 \leq i \leq m$ and $1 \leq j \leq p$ satisfying $\delta(i-1)<j \leq \delta(i)$.   Notice that $\Theta \mathcal C$ has a terminal object $[0]$.

\begin{example}
Take $\Theta_0$ be the category 1 with one object and the identity morphism only.  Inductively define $\Theta_n = \Theta \Theta_{n-1}$.  The categories $\Theta_n$ were first defined by Joyal; the inductive definition described here is due to Berger \cite{berger} and also given in \cite{rezktheta}.  Observe that $\Theta_1 = \Delta$.

Berger proved also that each $\Theta_n$ is a Reedy category, so the category $\sPsh(\Theta_n)$ can be equipped with the Reedy model structure, in which the weak equivalences are given by levelwise weak equivalences of simplicial structure \cite{berger2}.  In \cite{elegant}, we prove that $\Theta_n$ has the additional structure of an elegant Reedy category, so that the Reedy model structure on $\sPsh(\Theta_n)$  is exactly the injective model structure.
\end{example}

\setcounter{subsection}{1}

\subsection{Comparison functors involving $\Delta \times \mathcal C$ and $\Theta \mathcal C$}

Let $\mathcal C$ be a small category with terminal object $t=t_\mathcal C$, and let $p \colon \mathcal C \rightarrow 1$ be the unique functor.  We consider the categories $\Delta \times \mathcal C$ and $\Theta C$.  Both these categories have canonical functors to $\Delta$; in fact, we have adjoint pairs
\[ \pi_\Delta \colon \Delta \times \mathcal C \rightleftarrows \Delta \noloc \tau_\Delta \qquad \text{and} \qquad \pi_\Theta \colon \Theta \mathcal C \rightleftarrows \Theta 1 = \Delta \noloc \tau_\Theta \]
defined by $\pi_\Delta=\id_\Delta \times p$,
$\tau_\Delta=\id_\Delta\times t$, $\pi_\Theta=\Theta p$, and
$\tau_\Theta = \Theta t$.  Explicitly on objects, we have
\begin{align*}
\pi_\Delta(([m],c)) &= [m]  & \pi_\Theta([m](c_1,\dots,c_m)) &= [m] \\
\tau_\Delta([m]) &= ([m],t)  & \tau_\Theta([m]) &= [m](t,\dots,t).
\end{align*}

On the level of simplicial presheaves, these functors allow us to define an \emph{underlying simplicial space} associated to an object of $\sPsh(\Theta \mathcal C)$ or $\sPsh(\Delta \times \mathcal C)$; further evaluating this underlying simplicial space at degree zero gives an \emph{underlying space}.  

For example, consider a functor $X \colon \Deltaop \times \mathcal C^{op} \rightarrow \SSets$.  Then $\tau_\Delta^*(X)$ is the simplicial space given by $[m] \mapsto X([m], t)$, where $t$ is the terminal object of $\mathcal C$.  In other words, $\tau_\Delta^*$ is the simplicial space given by evaluation of $X$ at the terminal object of $\mathcal C$, also written as $X(-, t)$.  The underlying simplicial set of $X$ is given by $X([0], t)$.  

Let $d \colon \Delta \times \mathcal C \rightarrow \Theta \mathcal C$ be the functor defined on objects by
\[ d([m],c) := [m](c,\dots,c). \]
Note that $\pi_\Theta d = \pi_\Delta$ and $d \tau_\Delta = \tau_\Theta$.

The functor $d$ induces the sequence of adjoint functors on presheaves
\[\xymatrix{ {\sPsh(\Delta \times \mathcal C)}  \ar@<-5ex>[d]^{d_\#} \ar@<5ex>[d]^{d_*} \\
{\sPsh(\Theta \mathcal C).} \ar[u]_{d^*} }\]
These adjoint pairs define Quillen pairs
\[
d_\# \colon \sPsh(\Delta \times \mathcal C) \rightleftarrows \sPsh(\Theta \mathcal C) \noloc d^*\]
and
\[ d^*\colon \sPsh(\Theta \mathcal C) \rightleftarrows \sPsh(\Delta \times \mathcal C) \noloc d_* \]
on the projective and injective model structures, respectively.  While our primary interest is in the second Quillen pair, the first will be useful for making calculations.

\setcounter{theorem}{2}

\begin{prop}
There are natural isomorphisms $\pi_\Theta^* d^* \cong \pi_\Delta^*$ of functors $\sPsh(\Delta \times \mathcal C)\rightarrow \sPsh(\Delta)$, and natural
isomorphisms $\pi_\Delta^* d_\# \cong \pi_\Theta^*$ of functors $\sPsh(\Theta \mathcal C) \rightarrow \sPsh(\Delta)$.
\end{prop}

\begin{proof}
The first isomorphism follows immediately from the fact that $\pi_\Theta d =\pi_\Delta$.  The second follows from the fact that $d\tau_\Delta=\tau_\Theta$ and the observation that $\pi_\Delta^*\cong (\tau_\Delta)_\#$ and $\pi_\Theta^*\cong (\tau_\Theta)_\#$ from Section \ref{presheaf}.
\end{proof}

Explicitly, for any object $W$ of $\sPsh(\Delta \times \mathcal C)$, we have
\[ (d_\# W)([m](t,\dots,t)) \cong W([m],t), \]
and for any object $X$ of $\sPsh(\Theta \mathcal C)$, we have
\[ (d^*X)([m],t) \cong X([m](t,\dots,t)).\]
That is, the functors $d_\#$ and $d^*$ both preserve underlying simplicial spaces, up to isomorphism.

In particular, $(d_\# W)([0]) \cong W([0],t)$ and $(d^*X)([0],t) \cong X([0])$.  Since right adjoints preserve terminal objects, we know that $d^*(F_{\Theta \mathcal C} [0]) \cong F_\Delta[0]\times F_\mathcal C(t)$, from which we see that $(d^*W)([0]) \cong W([0],t)$.  That is, the functors $d_\#$, $d^*$, and $d_*$ all preserve underlying space, up to isomorphism.

\setcounter{subsection}{3}

\subsection{The intertwining functor}

Recall from \cite[4.4]{rezktheta} the \emph{intertwining functor}
\[ V\colon \Theta(\sPsh(\mathcal C)) \rightarrow \sPsh(\Theta \mathcal C), \]
defined on objects by
\begin{align*}
V[m](X_1,\dots,X_m)([k](c_1,\dots,c_k)) &
:= \Map_{\Theta(\sPsh(\mathcal C))}([k](Fc_1,\dots,Fc_k), [m](X_1,\dots,X_m)) \\
& \cong \coprod_{\delta \colon [k] \rightarrow [m]} \; \prod_{i=1}^k \; \prod_{j=\delta(i-1)+1}^{\delta(i)} X_j(c_i).
\end{align*}
The intertwining functor is the left Kan extension of
$F_{\Theta \mathcal C}$ along $\Theta F_\mathcal C$.

More generally, if $f \colon K \rightarrow F[m]$ is a map in $\SSets$, regarded as a maps of discrete objects in $\sPsh(\Theta \mathcal C)$, we define
\[ V_f [m](X_1,\dots,X_m) := \lim \left[ V[m](X_1,\dots,X_m) \rightarrow F[m] \xleftarrow{f} K\right],  \]
so that
\begin{align*}
V_f[m](X_1,\dots,X_m)([k](c_1,\dots,c_k)) & \cong \coprod_{\delta \colon F[k] \rightarrow K} \; \prod_{i=1}^k \; \prod_{j=f\delta(i-1)+1}^{f\delta(i)} X_j(c_i).
\end{align*}
If $K\subseteq F[m]$ is a subobject, we often write $V_K[m](X_1,\dots,X_m)$ for $V_f [m](X_1,\dots,X_m)$.

We recall that
\begin{multline*}
V_f[m](X_1,\dots,X_{d-1},\varnothing,X_{d+1},\dots,X_m)  \\
\cong V_{f \times_{F[m]} F[0, d-1]}[m](X_1,\dots,X_{d-1}) \amalg V_{f \times_{F[m]} F[d,m]}[m](X_{d+1},\dots X_m),
\end{multline*}
and that
\[ V_f[m](X_1,\dots,X_{d-1},{-},X_{d+1},\dots,X_m) \colon \sPsh(\mathcal C) \rightarrow V_f[m](X_1,\dots,X_{d-1},\varnothing,X_{d+1},\dots, X_m) \backslash \sPsh(\Theta \mathcal C) \]
is a left adjoint (and thus is colimit preserving), where the category on the right is the category of objects of $\sPsh(\Theta \mathcal C)$ under $V_f[m](X_1,\dots,X_{d-1},\varnothing,X_{d+1},\dots, X_m)$.  In particular, the functor
\[ V[1] \colon \sPsh(\mathcal C) \rightarrow (F[0]\amalg F[0]) \backslash \sPsh(\Theta \mathcal C) \]
is a left adjoint.  Its right adjoint $(\partial F[1] \xrightarrow {(x_0,x_1)} X) \mapsto M^\Theta_X(x_0,x_1)$ is
described below.    A short calculation gives the following useful fact.

\setcounter{theorem}{4}

\begin{prop}
For any object $c$ of $\mathcal C$, there is an isomorphism $V[1](F_\mathcal C(c)) \cong F_\Theta([1](c))$ in $\sPsh(\Theta \mathcal C)$.
\end{prop}

We also have that if $f \colon K \rightarrow F[m]$ is a colimit of a diagram $J \rightarrow \SSets / F[m]$ which sends an object $j$ of $J$ to a map $f_j \colon K_j \rightarrow F[m]$, we have
\[ V_f[m](X_1,\dots,X_m) \cong \colim_{j\in J} V_{f_j}[m](X_1,\dots,X_m), \]
the latter colimit taking place in $\sPsh(\Theta \mathcal C)$.

\setcounter{subsection}{5}

\subsection{Relation between the comparison and intertwining functors} \label{intertwiningsec}

The results in this section are quite technical, but will be useful in what follows.

It is important to understand $d^*(V_K[m](X_1,\dots,X_m))$, where
$K \subseteq F[m]$.  We have
\begin{align*}
d^*(V_K[m](X_1,\dots,X_m)) ([k],c) & \cong \prod_{\delta \colon F[k] \rightarrow K} \; \prod_{i=1}^k \; \prod_{j=\delta(i-1)+1}^{\delta(i)} X_j(c) \\
& \cong \prod_{\delta \colon F[k] \rightarrow K} \; \prod_{j=\delta(0)+1}^{\delta(k)} X_j(c).
\end{align*}
In particular,
\[ d^*(V[0]) \cong F[0]=1, \]
and
\[ \begin{aligned}
d^*(V[1](X)) & \cong \colim \left(\coprod_{\delta \colon [k] \rightarrow [1]} X(c) \leftarrow X(c) \amalg X(c) \rightarrow F[0] \amalg F[0] \right) \\
& \cong F[1] \times X \cup_{\partial F[1] \times X} \partial F[1],
\end{aligned} \]
where $F[0]\amalg F[0] \cong \partial F[1]\subseteq F[1]$ is the usual subobject.
We write $\Sigma \colon \sPsh(\mathcal C) \rightarrow \sPsh(\Delta\times \mathcal C)$ for the functor $\Sigma(X):=d^*(V[1](X))$; it comes with a distinguished map
$\partial F[1] \rightarrow \Sigma(X)$.

\setcounter{theorem}{6}

\begin{prop} \label{prop:d-lowerhash-of-suspension}
There is a natural isomorphism $d_\#(\Sigma X) \cong V[1](X)$, compatible with the natural isomorphism $d_\#(\partial F[1]) \cong
V[0] \amalg V[0]$.
\end{prop}

\begin{proof}
The functors
\[ V[1], d_\# \Sigma \colon \sPsh(\mathcal C) \rightarrow (F[0] \amalg F[0]) \backslash \sPsh(\Theta \mathcal C) \]
preserve $\SSets$-weighted colimits.  The isomorphism $\Sigma \xrightarrow{\cong} d^*V[1]$ is adjoint to a map
$\alpha\colon d_\# \Sigma \rightarrow V[1]$.  Thus, it suffices to check that $\alpha$ is an isomorphism at representable objects, which holds since there are isomorphisms
\[ \begin{aligned}
d_\# \Sigma(Fc) & \cong d_\# (F[1] \times Fc \cup_{\partial F[1]\times Fc} \partial F[1]) \\
& \cong F[1](c) \cup_{F[0]\amalg F[0]}(F[0] \amalg F[0]) \cong F[1](c). 
\end{aligned} \]
\end{proof}

The following result gives us an inductive description of $d^*(V_f[m](X_1,\dots,X_m))$, as a quotient of $K\times (X_1\times
\cdots\times X_m)$.

\begin{prop} \label{prop:inductive-d-upperstar-v}
Fix a map $f \colon K \rightarrow F[m]$ of simplicial sets.  Let
\[ T_m=F[0,m-1]\cup_{F[1,m-1]} F[1,m], \]
which comes with an inclusion $T_m \rightarrow F[m]$; the object $T_m$ is the union of the first and last faces of the $m$-simplex along their common intersection.  Then for $m \geq 1$ there is a
pushout diagram in $\sPsh(\Delta\times \mathcal C)$ of the form
\[\xymatrix{ {(T_m\times_{F[m]}K)\times (X_1\times \cdots \times X_m)} \ar[r] \ar[d] & {K\times (X_1\times \cdots\times X_m)} \ar[d] \\
{d^*(V_{T_m\times_{F[m]} K}[m](X_1,\dots,X_m))} \ar[r] & {d^*(V_K[m](X_1,\dots,X_m)).} }\]
\end{prop}

\begin{proof}
Assume for simplicity that $f=\id_{F[m]}$.  Note that maps $\delta \colon F[k] \rightarrow T_m$ are precisely the maps $\delta \colon [k] \rightarrow [m]$ such that either
$\delta(0)>0$ or $\delta(k)<m$.  If we
evaluate the above square at each corner at $([k],c)$, we get
\[\xymatrix{ {\coprod_{\delta \colon F[k] \rightarrow T_m} \; \prod_{j=1}^m X_j(c)} \ar[r] \ar[d]
& {\coprod_{\delta \colon F[k] \rightarrow F[m]} \; \prod_{j=1}^m X_j(c)} \ar[d] \\
{\coprod_{\delta \colon F[k] \rightarrow T_m} \; \prod_{j=\delta(0)+1}^{\delta(k)} X_j(c)} \ar[r] &
{\coprod_{\delta \colon F[k] \rightarrow F[m]}\; \prod_{j=\delta(0)+1}^{\delta(k)} X_j(c)} }\]
which is a pushout.
\end{proof}

Recall that because 
\[ T_m=F[0,m-1]\cup_{F[1,m-1]}F[1,m], \] 
we have that 
\[ V_{T_m} \cong V_{F[0,m-1]}\cup_{V_{F[1,m-1]}} V_{F[1,m]}, \] 
and $d^*$ preserves colimits.  Thus the above proposition gives an inductive description of $d^*V_f$.

We note that since $T_m \rightarrow F[m]$ is a monomorphism, it follows that the above square is actually a homotopy pushout with respect to levelwise weak equivalences in $\sPsh(\Delta \times \mathcal C)$.  Furthermore, since 
\[ T_m=F[0,m-1]\cup_{F[1,m-1]}F[1,m] \] 
presents $T_m$ as a pushout along inclusions, the resulting pushout squares building 
\[ (T_m\times_{F[m]}f)\times (X_1\times\cdots \times X_m) \] 
and
\[ d^*(V_{T_m \times_{F[m]} f}[m](X_1,\dots,X_m)) \] 
are also homotopy pushouts.

\setcounter{subsection}{8}

\subsection{Mapping objects} \label{mapping}

Given $X$ in $\sPsh(\Theta \mathcal C)$ and a vertex $(x_0,\dots,x_m) \in X[0]^{m+1}$, we would like to define the mapping object of $X$ associated to the vertices $x_0, \ldots, x_m$.  When $m=1$, this object gives the presheaf of functions from $x_0$ to $x_1$, but we give a more general definition.  Let $M_X^\Theta(x_0,\dots,x_m)$ denote the object of $\sPsh(\mathcal C^m)$ defined by
\[ M_X^\Theta(x_0,\dots x_m)(c_1,\dots,c_m) := \lim \left(X[m](c_1,\dots,c_m) \rightarrow X[0]^{m+1} \leftarrow \{(x_0,\dots,x_m)\} \right). \]
Equivalently, $M^\Theta_X(x_0,\dots,x_m)(c_1,\dots,c_m)$ is the fiber of the map
\[ \Map_{\sPsh(\Theta \mathcal C)}(F[m](c_1,\dots,c_m) , X) \rightarrow \Map_{\sPsh(\Theta \mathcal C)}(\amalg_{m+1}F[0], X) \] 
over $(x_0,\dots,x_m)$.

Likewise, given $W$ in $\sPsh(\Delta \times \mathcal C)$ and $(x_0,\dots,x_m) \in W([0],t)^{m+1}$,
let $M_W^\Delta(x_0,\dots,x_m)$ denote the object of $\sPsh(\mathcal C)$ defined by
\[ M_W^\Delta(x_0,\dots,x_m)(c) := \lim \left( W([m],c) \rightarrow  W([0],c)^{m+1} \leftarrow \{(x_0,\dots,x_m)\} \right) \]
where the second arrow uses the map $W([0],t) \rightarrow W([0],c)$.

In particular, we note that  $M_W^\Delta(x_0,x_1)(c)$ is the fiber of the map
\[ \Map_{\sPsh(\Delta\times \mathcal C)}(\Sigma(F_\mathcal C c), W) \rightarrow \Map_{\sPsh(\Delta \times \mathcal C)} (F[0] \amalg F[0], W). \]

The following fact can be verified from the definitions of these mapping objects.

\setcounter{theorem}{9}

\begin{prop}
The functors $d^*$ and $d_*$ are compatible with $M^\Delta$ and $M^\Theta$.  That is, if $W$ is an object of $\sPsh(\Delta\times \mathcal C)$ and $X$ is an object of $\sPsh(\Theta \times \mathcal C)$, then
\[ M^\Theta_{d_*W}(x_0,x_1) \cong M^\Delta_W(x_0,x_1) \quad \text{and} \quad M^\Delta_{d^*X}(y_0,y_1) \cong M^\Theta_X(y_0,y_1) \]
for all $x_0,x_1\in W([0],t) \cong (d_*W)([0])$ and $y_0,y_1 \in X([0]) \cong (d^*X)([0],t)$.
\end{prop}

\section{$\Theta$-objects}

Let $\mathcal C$ be a small category with terminal object $t$. In this section, we are interested in a certain localization of $\sPsh(\Theta \mathcal C)$ which is described in more detail in \cite{rezktheta}.  The local objects are those satisfying Segal and completeness conditions, together with local conditions inherited from a localization of $\sPsh(\mathcal C)$, given by a set $\mathcal{S}$ of cofibrations in $\sPsh(\mathcal C)$.

Recall that for simplicial spaces, the inclusion maps $\varphi^m \colon G[m] \rightarrow F[m]$ for $m \geq 0$ are used to encode the Segal condition, where
\[ G[m] = \colim\left(F[1] \rightarrow F[0] \leftarrow \cdots \rightarrow F[0] \leftarrow F[1] \right) = F[0,1]\cup\cdots\cup F[m\!-\!1,m]. \] 
Notice that $\varphi^0$ and $\varphi^1$ are isomorphisms.  Additionally, if $E$ denotes the discrete nerve of the groupoid with two objects and a single isomorphism between them, the collapse map $z \colon E \rightarrow F[0]$ is used to encode the completeness condition.

For the more general case of $\sPsh(\Theta \mathcal C)$, the Segal condition is similarly encoded by maps
\[ G[m](c_1, \ldots, c_m) \rightarrow F[m](c_1, \ldots, c_m) \]
for each $m \geq 0$ and $m$-tuple $(c_1, \ldots, c_m)$ of objects of $\mathcal C$, where
\[ G[m](c_1, \ldots, c_m) := \colim \left(F[1](c_1) \rightarrow F[0] \leftarrow \cdots \rightarrow F[0] \leftarrow F[1](c_m)\right). \]
Denote by $\Se^\Theta$ the set of all such maps. Equivalently, we can use the intertwining functor to define $\Se^\Theta$ as the set of all maps of the form
\[ V_{G[m]}[m](c_1, \ldots, c_m)  \rightarrow V[m](c_1, \ldots, c_m). \]

\begin{definition} \label{thetaobject}
Let $(\mathcal C, \mathcal S)$ be a presentation.  An object $X$ of $\sPsh(\Theta \mathcal C)$ is a $\Theta$-\emph{object} relative to $\mathcal{S}$ if:
\begin{enumerate}
\item [(1)] $X$ is injective fibrant,
\item [(2)] for all $m\geq2$ and all tuples $c_1,\dots,c_m$ of objects in $C$, the map
\[ (X\delta^{01}, \dots, X\delta^{m-1,m})\colon X[m](c_1,\dots,c_m) \rightarrow X[1](c_1)\times_{X[0]} \cdots\times_{X[0]} X[1](c_m) \]
is a weak equivalence of simplicial sets,
\item [(3)] the object $\tau_{\Theta}^*X$ of $\sPsh(\Delta)$ is a complete Segal space, and
\item [(4)] for all $x_0,x_1\in X[0]$, the object $M_X^\Theta(x_0,x_1)$ in $\sPsh(\mathcal C)$ is $\mathcal{S}$-local.
\end{enumerate}
\end{definition}

\begin{remark}
We note that, in the presence of (1)  and (4), we can replace condition (2) with the following.
\begin{enumerate}
\item [(2')] For all $m\geq 2$, all tuples $c_1,\dots,c_m$ of objects in $C$, and all $x_0,\dots,x_m\in X[0]$, the map
\[ M_X^\Theta(x_0,\dots,x_m) \rightarrow M_X^\Theta(x_0,x_1)\times \cdots \times M_X^\Theta(x_{m-1},x_m).  \]
is a weak equivalence.
\end{enumerate}
\end{remark}

The following result was shown in \cite[\S 8]{rezktheta}.

\begin{prop}
Consider the injective model structure $\sPsh(\Theta \mathcal C)$, and the following sets of maps.
\begin{enumerate}
\item [(1)] Let $\Se^\Theta=\{ V_{G[m]}[m](c_1,\dots,c_m) \rightarrow V[m](c_1,\dots,c_m) \}$, where $m \geq 0$ and $(c_1,\dots,c_m) \in \ob(\mathcal C)^m$.
\item  [(2)] Let $\Cpt^\Theta = \{ \pi_\Theta^*z \}$, where $z \colon E \rightarrow F[0]$ in $\sPsh(\Delta)$ is the map defining completeness for simplicial spaces.
\item [(3)] Let $\Rec^\Theta(\mathcal {S}) = \{ V[1](f) \}$, where $f$ ranges over all elements of $\mathcal{S}$.
\end{enumerate}
Let $\mathcal{S}^\Theta = \Se^\Theta \cup \Cpt^\Theta \cup \Rec^\Theta(\mathcal{S})$.  Then $X$ is a $\Theta$-object if and only if it is injective fibrant and $\mathcal{S}^\Theta$-local.
\end{prop}

The names of the sets of maps are meant to suggest, respectively, Segal maps, the completeness map, and recursive maps, which capture localizations from $\sPsh(\mathcal C)_\mathcal S$.  The following theorem tells us that these maps define cartesian presentations.

\begin{theorem} \cite[6.1, 7.21, 8.1]{rezktheta}
The presentations $(\Theta \mathcal C, \Se^\Theta)$ and $(\Theta \mathcal C, \Se^\Theta\cup \Cpt^\Theta)$ are cartesian.  Furthermore, if $(\mathcal C, \mathcal{S})$ is cartesian, then so is $(\Theta \mathcal C, \mathcal{S}^\Theta)$.
\end{theorem}

The following criterion will be useful in what follows.

\begin{prop}\label{prop:we-local-objects-theta}
Let $f \colon X \rightarrow Y$ be a map between $\mathcal{S}^\Theta$-local objects in $\sPsh(\Theta \mathcal C)$.  Then $f$ is in $\overline{\mathcal{S}^\Theta}$ if and only if the map 
\[ f([0]) \colon X[0] \rightarrow Y[0] \]
and the maps
\[ f[1](c) \colon X[1](c) \rightarrow Y[1](c) \]
for all objects $c$ of $\mathcal C$ are weak equivalences of spaces.
\end{prop}

\begin{proof}
Recall that a map between $\mathcal{S}^\Theta$-fibrant objects is in $\overline{\mathcal{S}^\Theta}$ if and only if it is a levelwise weak
equivalence.  For $\mathcal{S}^\Theta$-fibrant objects, the value at a general $[m](c_1,\dots,c_m)$ in $\Theta \mathcal C$ is a homotopy limit of values at $[1](c_i)$, for each $1 \leq i \leq m$, and $[0]$, verifying the proposition when $X$ and $Y$ are $\mathcal{S}^\Theta$-fibrant.  Since any $\mathcal{S}^\Theta$-local object is levelwise weakly equivalent to a $\mathcal{S}^\Theta$-fibrant object, the proposition follows.
\end{proof}

\begin{example}
When $\mathcal C = \Theta_{n-1}$, so that $\Theta \mathcal C = \Theta_n$, then $\Theta$-objects are called $\Theta_n$-\emph{spaces}.  Since the categories $\Theta_n$ are meant to encode higher categories, in the sense that $\Delta$ encodes categories, $\Theta_n$-spaces are higher-order versions of complete Segal spaces and hence models for $(\infty, n)$-categories.
\end{example}

\section{Segal objects and complete Segal objects in $\sPsh(\Delta \times \mathcal C)$}

Again, let $\mathcal C$ be a small category with terminal object $t=t_\mathcal C$.  We would like to consider objects in $\sPsh(\Delta \times \mathcal C)$ which are analogous to the $\Theta$-objects in $\sPsh(\Theta \mathcal C)$.

\begin{definition}
Let $(\mathcal C,\mathcal S)$ be a presentation.  An object $W$ of $\sPsh(\Delta \times \mathcal C)$ is a \emph{Segal object relative to} $\mathcal S$ if
\begin{enumerate}
\item \label{injfib} the object $W$ is injective fibrant.
\item \label{segal} for all $m \geq 2$ and $c$ objects of $\mathcal C$, the map
\[ (W \delta^{01},\cdots, W \delta^{m-1,m}) \colon W([m],c) \rightarrow W([1],c)\times_{W([0],c)} \cdots \times_{W([0],c)}  W([1],c) \]
is a weak equivalence of spaces, and

\item \label{recursive} for all $x_0,x_1\in W([0],t)$, the object $M_W^\Delta(x_0,x_1)$ in $\sPsh(\mathcal C)$ is $\mathcal{S}$-local.
\end{enumerate}
\end{definition}

We prove the following theorem in \cite[3.14]{part1} in the case where $\mathcal C= \Theta_n$ with the $\Theta$-object localization.  Here, we consider the more general case.

\begin{theorem} \label{ssthetansp}
There is a localization of the injective model structure on $\sPsh(\Delta \times \mathcal C)$ in which the fibrant objects are precisely the Segal objects.  Furthermore, this model structure is cartesian if $(\mathcal C, \mathcal S)$ is a cartesian presentation.
\end{theorem}

\begin{proof}
The existence of the model structure is obtained using a left Bousfield localization of the Reedy model structure \cite[4.1.1]{hirsch}.
We localize with respect to the union of the following two sets of morphisms:
\begin{enumerate}
\item $\Se^\Delta=\{ G[m] \times Fc \rightarrow F[m]\times Fc \}$, for all $c \in \ob(\mathcal C)$ and $m \geq 2$, and

\item $\Rec^\Delta(\mathcal{S}) = \{ \Sigma(s) \colon \Sigma(S) \rightarrow \Sigma(S') \}$
for all $(s\colon S \rightarrow S') \in \mathcal{S}$.
\end{enumerate}

An object $W$ which is local with respect to $\Se^\Delta$ satisfies the condition that $W(-, c)$ is a Segal space for all objects $c$ of $\mathcal C$, and if $W$ is local with respect to $\Rec^\Delta$ then $W([m], -)$ is an $\mathcal S$-local object in $\sPsh(\mathcal C)$ for all $[m] \in \ob(\Delta)$.

To prove that the model structure is cartesian when $(\mathcal C, \mathcal S)$ is, we apply Proposition \ref{cartesian} to $\Delta \times \mathcal C$ and use the localization just described.  Therefore, we must prove that if $W$ is a Segal object, then so is $W^{F_{\Delta \times \mathcal C}([m], c)}$ for any object $([m], c)$ of $\Delta \times \mathcal C$.

Since $W$ is assumed to be a Segal object, we know that, for any object $[m]$ of $\Delta$, $W([m], -)$ is a $\mathcal S$-local, and, for any object $c$ of $\mathcal C$, $W(-, c)$ is a Segal space.  Since the Segal space model structure and $\sPsh(\mathcal C)_{\mathcal S}$ are cartesian, we know that $W(-, c)^{F_\Delta[m]}$ is a Segal space and $W([m], -)^{F_\mathcal C(c)}$ is $\mathcal S$-local.  Therefore, it suffices to prove that
\[ W^{F_{\Delta \times \mathcal C}([m], t)}(-, c) \cong W(-, c)^{F_\Delta[m]} \]
 for every object $[m]$ of $\Delta$ and the terminal object $t$ of $\mathcal C$, and
\[ W^{F_{\Delta \times \mathcal C}([0], c)}([m], -) \cong W([m], -)^{F_\mathcal C(c)} \]
for every object $c$ of $\mathcal C$.
We prove the first of these isomorphisms; the second is proved analogously.

Let $\Mapobj_\Delta$, $\Mapobj_\mathcal C$, and $\Mapobj_{\Delta \times \mathcal C}$ denote the internal hom objects in $\sPsh(\Delta)$, $\sPsh(\mathcal C)$, and $\sPsh(\Delta \times \mathcal C)$, respectively. Then we have isomorphisms on the level of $k$-simplices
\[ \begin{aligned}
\left[W(-, c)^{F_\Delta[m]} \right]_k & \cong \Mapobj_\Delta(F_\Delta[k], W(-, c)^{F_\Delta[m]})_0 \\
& \cong \Mapobj_\Delta(F_\Delta[k] \times F_\Delta[m], W(-, c))_0 \\
& \cong \Mapobj_\Delta(F_\Delta[k] \times F_\Delta[m], ([\ell] \mapsto \Map_\mathcal C(F_\mathcal C(c), W([\ell], -))(t)))_0 \\
& \cong \Mapobj_{\Delta \times \mathcal C}(F_\Delta[k] \times F_\Delta[m] \times F_\mathcal C(c), W)([0], t) \\
& \cong \Mapobj_{\Delta \times \mathcal C}(F_\Delta[k] \times F_\mathcal C(c), W^{F_\Delta[m]})([0], t) \\
& \cong \Mapobj_{\Delta \times \mathcal C}(F_\Delta[k] \times F_\mathcal C(c), W^{F_{\Delta \times \mathcal C}([m], t)})([0], t) \\
& \cong \Mapobj_\Delta(F_\Delta[k], W^{F_{\Delta \times \mathcal C}([m], t)}(- , c))_0 \\
& \cong \left[W^{F_{\Delta \times \mathcal C}([m], t)} (-, c)\right]_k
\end{aligned} \]
which proves the desired isomorphism.
\end{proof}

We now add further conditions to obtain complete Segal objects.

\begin{definition} \label{cssobj}
Let $(\mathcal C,\mathcal S)$ be a presentation.  An object $W$ of $\sPsh(\Delta \times \mathcal C)$ is a \emph{complete Segal object} relative to $\mathcal S$ if:
\begin{enumerate}
\item \label{injfib} the object $W$ is injective fibrant,

\item \label{segal} for all $m \geq 2$ and $c \in \ob(\mathcal C)$, the map
\[ (W \delta^{01},\cdots, W \delta^{m-1,m}) \colon W([m],c) \rightarrow W([1],c)\times_{W([0],c)} \cdots \times_{W([0],c)}  W([1],c) \]
is a weak equivalence of spaces,

\item \label{localmapping} for all $x_0,x_1\in W([0],t)$, the object $M_W^\Delta(x_0,x_1)$ in $\sPsh(\mathcal C)$ is $\mathcal{S}$-local,

\item \label{complete0} the object $\tau_\Delta^* W$ of $\sPsh(\Delta)$ is a complete Segal space, and

\item \label{essconst} for all objects $c \in \mathcal C$, the map $W([0],t) \rightarrow W([0],c)$ is a weak equivalence.
\end{enumerate}
\end{definition}

\begin{remark}
The need for condition (\ref{essconst}) requires some explanation, since it is unnecessary in the case of complete Segal spaces.  A simplicial object in any category $\mathcal C$, satisfying the Segal condition, models an internal category $\mathcal C$, where both the morphisms and the objects have the structure of objects of $\mathcal C$.  However, our objective in defining complete Segal objects is to have an up-to-homotopy version of categories enriched in $\mathcal C$, where only the morphisms are objects of $\mathcal C$ and the objects form a set.

This feature is more distinct in the structure of Segal category objects, where the degree zero object is forced to be a set.  Looking for a moment at the case of $(\infty, 1)$-categories, the transition from Segal categories (with discrete space in degree zero) to complete Segal spaces (where we drop the discreteness assumption but require completeness) can be thought of as a homotopical change: a set is replaced by a space.  Here, we want the same transition.  In particular, we want the object in degree zero to be a space, not a more general object of $\sPsh(\mathcal C)$.
\end{remark}

\begin{remark} \label{4vs4'}
Condition (\ref{complete0}) is equivalent to requiring that the simplicial space $W(-, t)$ be a complete Segal space, where $t$ is the terminal object of $\mathcal C$.  In other words, this assumption corresponds to the desired weak equivalence between the object in degree zero and the object of homotopy equivalences.  We do not assume what one might expect, which is the following condition:
\begin{itemize}
\item[(4')] $W(-, c)$ is a complete Segal space for any object $c$ in $\mathcal C$.
\end{itemize}
This condition is, however, the more desirable one, in that it is compatible with the cartesian structure.  As we will see, in some cases the two conditions are equivalent.
\end{remark}

\begin{theorem} \label{completemc}
There is a simplicial model structure on $\sPsh(\Delta \times \mathcal C)$ in which the fibrant objects are the complete Segal space objects.
\end{theorem}

\begin{proof}
Consider the following sets of maps in $\sPsh(\Delta \times \mathcal C)$.
\begin{itemize}
\item Let $\Se^\Delta=\{ G[m] \times Fc \rightarrow F[m]\times Fc \}$, where $c$ ranges over objects of $\mathcal C$, and $m \geq 2$.

\item Let $\Rec^\Delta(\mathcal{S}) = \{ \Sigma(s) \colon \Sigma(S) \rightarrow \Sigma(S') \}$
for all $(s\colon S \rightarrow S') \in \mathcal{S}$.

\item Let $\Cpt^\Delta=\{ \pi_\Delta^*z \}$, where $z \colon E \rightarrow F[0]$ is the map defining completeness in $\sPsh(\Delta)$.

\item Let $\Coll^\Delta = \{ F([0],c)\rightarrow F([0],t) \}$ for all $c$ in $C$.
\end{itemize}

The naming is done analogously to $\Theta$-objects; the name of the last set is meant to suggest collapse maps.  Let $\mathcal{S}^\Delta= \Se^\Delta\cup \Cpt^\Delta \cup \Rec^\Delta(\mathcal{S}) \cup \Coll^\Delta$.  Then $X$ is a complete Segal object if and only if it is injective fibrant and $\mathcal{S}^\Delta$-local.

Localizing the injective model structure with respect to $\mathcal S^\Delta$ produces the desired model structure.
\end{proof}

The following result allows us to identify complete Segal objects.

\begin{prop}\label{prop:we-local-objects-delta}
Let $f \colon W \rightarrow Z$ be a map between $\mathcal{S}^\Delta$-local objects in $\sPsh(\Delta \times \mathcal C)$.  Then $f$ is in $\overline{\mathcal{S}^\Delta}$  if and only if
\[ f([0],t) \colon W([0],t) \rightarrow Z([0],t) \]
and
\[ f([1],c) \colon W([1],c) \rightarrow Z([1],c), \]
for all objects $c$ of $C$, are weak equivalences of spaces.
\end{prop}

\begin{proof}
Recall that a map between $\mathcal{S}^\Delta$-fibrant objects is in $\overline{\mathcal{S}^\Delta}$ if and only if it is a levelwise weak equivalence.  For $\mathcal{S}^\Delta$-fibrant objects, the value at a general $([m],c)$ in $\Delta \times \mathcal C$ is a homotopy limit of values at $([1],c)$ and $([0],t)$, so the proposition holds when $W$ and $Z$ are $\mathcal{S}^\Delta$-fibrant.  Since any $\mathcal{S}^\Delta$-local object is levelwise weakly equivalent to a $\mathcal{S}^\Delta$-fibrant object, the proposition follows.
\end{proof}

We now turn to model structures on $\sPsh(\Delta \times \mathcal C)$ where only some of the conditions for complete Segal objects hold.

\begin{prop}  \label{almostcomplete}
There is a simplicial model structure $\mathcal{ACSS}$ on the category $\sPsh(\Delta \times \mathcal C)$ in which the fibrant objects satisfy conditions \eqref{injfib}-\eqref{complete0}.
\end{prop}

\begin{proof}
The existence of the model structure can be obtained by localizing the Segal object model structure with respect to $\Cpt^\Delta$.
\end{proof}

We call the fibrant objects in this model structure \emph{almost complete Segal objects}. However, this model structure as described is not cartesian.  Therefore, we include the following result.

\begin{prop} \label{123'}
There is a cartesian model structure on the category $\sPsh(\Delta \times \Theta_n)$ in which the fibrant objects satisfy conditions \eqref{injfib}-\eqref{recursive} and (4').
\end{prop}

\begin{proof}
The model structure can be obtained by localizing the Segal object model structure with respect to the set of maps
\[ \{F[0] \times Fc \rightarrow E \times Fc \} \]
where $c$ ranges over all objects of $\mathcal C$.
The proof that the model structure is cartesian can proved similarly as the Segal object model structure was in Theorem \ref{ssthetansp}.
\end{proof}

These two model structures actually coincide in the special case when $\mathcal C=\Theta_n$.  The following proposition is a generalization of a result of Johnson-Freyd and Scheimbauer \cite[2.8]{jfs}; while their proof is given in the context of $\sPsh(\Delta^n)$, their argument can be modified to our context.

\begin{prop} \label{4implies4'}
An object in $\sPsh(\Delta \times \Theta_n)$ which satisfies conditions \eqref{injfib}-\eqref{complete0} also satisfies the stronger condition (4').
\end{prop}

\section{Comparison between complete Segal objects and $\Theta$-objects}

In this section we establish a Quillen equivalence between the model structures for $\Theta$-objects and the model structure for complete Segal objects.

\begin{prop}\label{prop:d-lowerstar-fibrant}
The functor $d_*\colon \sPsh(\Delta \times \mathcal C) \rightarrow \sPsh(\Theta \mathcal C)$ takes $\mathcal{S}^\Delta$-fibrant objects to $\mathcal{S}^\Theta$-fibrant objects.
\end{prop}

\begin{proof}
Given a complete Segal object $W$, we must check conditions (1),(2), (3), and (4) of Definition \ref{thetaobject} for $d_*W$ to be a $\Theta$-object.   Since $d_*$ is the right adjoint of a Quillen pair between injective model categories, $d_*W$ is injective fibrant, establishing (1).

Recall that $\Map_{\sPsh(\Theta \mathcal C)}(X, d_*W) \cong \Map_{\sPsh(\Delta \times \mathcal C)}(d^*X, W)$.  Thus, to show that (2), (3), and (4) hold, it suffices to show that $d^*\Se^\Theta$, $d^*\Cpt^\Theta$, and $d^*\Rec^\Theta(\mathcal{S})$ are all contained in $\overline{\mathcal{S}^\Delta}$.

We note that to show that a map $f$ is contained in $\overline{\mathcal{S}^\Delta}$, it suffices to show that $f$ is weakly equivalent (with respect to levelwise weak equivalences in $\sPsh(\Delta\times \mathcal C)$) to a homotopy colimit $\hocolim f_\alpha$, where each $f_\alpha \in \overline{\mathcal S^\Delta}$, and the homotopy colimit is computed with respect to the levelwise weak equivalences.

We further note that if $B \xleftarrow{f} A \rightarrow C$ is a diagram in any of the categories $\sPsh(\Delta \times \mathcal C)$, $\sPsh(\Theta \mathcal C)$, or $\sPsh(\Delta)$ such that $f$ is an injective cofibration (i.e., a levelwise monomorphism), then the pushout of the diagram is also a homotopy pushout with respect to levelwise weak equivalences.  Furthermore, $d^*$ preserves such diagrams, being the left adjoint of a Quillen pair between injective model structures.

To prove (2), we want to show that $d^*\Se^\Theta \subseteq \overline{\mathcal S^\Delta}$.  Consider the maps
\[ \varphi^{m,\underline{c}}=V_{\varphi^m}[m](Fc_1,\dots,Fc_m)\colon V_{G[m]} [m](Fc_1,\dots,Fc_m) \rightarrow V[m](Fc_1,\dots,Fc_m) \]
in $\Se^\Theta$.
We aim to show by induction on $m$ that $d^*\varphi^{m,\underline{c}} \in \overline{\mathcal S^\Delta}$ for all $\underline{c}=c_1,\dots,c_m$.  The cases of $m=0$ and $m=1$ are immediate, since $\varphi^0$ and $\varphi^{1,c}$ are isomorphisms.

Now suppose $m \geq 2$.  We need to show that $d^*\varphi^{m,\underline{c}} \in \overline{\mathcal S^\Delta}$ by using the (homotopy) pushout square of Proposition \ref{prop:inductive-d-upperstar-v}.  That is, we need to show that each of the three maps
\[ \alpha=(T_m\times_{F[m]} \varphi^m)\times (Fc_1\times\cdots \times Fc_m), \]
\[ \beta=\varphi^m\times (Fc_1\times \cdots \times Fc_m), \]
and
\[ \gamma=d^*(V_{T_m\times_{F[m]}s_m}[m](Fc_1,\dots,Fc_m)). \]
are contained in $\overline{\mathcal S^\Delta}$.   (Recall that $T_m=F[0,m-1]\cup_{F[1,m-1]}F[1,m]$.)

Because $\Se^\Delta$ is cartesian, we know that $\varphi^k \times Y \in \overline{\Se^\Delta} \subseteq \overline{\mathcal S^\Delta}$ for all $k \geq 0$ and all $Y$ in $\sPsh(\Delta \times \mathcal C)$.   Thus the map $\beta$ is in $\overline{\mathcal S^\Delta}$.

It is straightforward to check that the pullback of the map $\varphi^m\colon G[m] \rightarrow F[m]$ along the inclusion $F[p,p+k] \rightarrow F[m]$ is isomorphic to the map $\varphi^k \colon G[k] \rightarrow F[k]$.  We can use this observation to identify both $\alpha$ and $\gamma$ as homotopy pushouts of maps in $\overline{\mathcal S^\Delta}$.

That is, the map $\alpha$ is itself a (homotopy) pushout of maps of the form $\varphi^k \times Y$, and thus is in $\overline{\mathcal S^\Delta}$.   Similarly, the map $\gamma$ is a (homotopy) pushout of maps of the form $V_{\varphi^k}[k](Fc_{i_1},\dots,Fc_{i_k})$, which are contained in $\overline{\mathcal S^\Delta}$ by the inductive hypothesis, since $k<m$.

To prove (3), recall that $d^*\pi_\Theta^*=\pi_\Delta^*$, and thus $d^*\Cpt^\Theta = \Cpt^\Delta \subseteq \overline{\mathcal S^\Delta}$.  Lastly, to prove (4), recall from Section \ref{intertwiningsec} that $d^*V[1](f)=\Sigma(f)$, and thus $d^*\Rec^\Theta(\mathcal S)\subseteq \Rec^\Delta(\mathcal S) \subseteq \overline{\mathcal S^\Delta}$.
\end{proof}

\begin{cor} \label{dquillenpair}
The adjoint pair
\[ d^*\colon \sPsh(\Theta \mathcal C)_{\mathcal S^\Theta} \rightleftarrows \sPsh(\Delta \times \mathcal C)_{\mathcal S^\Delta}\noloc d_* \]
is a Quillen pair.
\end{cor}

\begin{proof}
Using Proposition \ref{prop:d-lowerstar-fibrant}, we can conclude that $d^*(\overline{\mathcal S^\Theta}) \subseteq \overline{\mathcal S^\Delta}$, and thus $d^*$ preserves both cofibrations and acyclic cofibrations.
\end{proof}

\begin{prop} \label{prop:d-upperstar-local}
The functor $d^*\colon \sPsh(\Theta \mathcal C) \rightarrow \sPsh(\Delta \times \mathcal C)$ takes $\mathcal S^\Theta$-local objects to $\mathcal S^\Delta$-local objects.
\end{prop}

\begin{proof}
Given an $\mathcal S^\Theta$-local object $X$ of $\sPsh(\Theta \mathcal C)$, we want to show that $d^*X$ is $\mathcal S^\Delta$-local.  Since $d^*$ preserves all levelwise weak equivalences, we can assume without loss of generality that $X$ is also injective fibrant, i.e., that $X$ is $\mathcal S^\Theta$-fibrant.  It suffices to show that an injective fibrant replacement $f\colon d^*X \rightarrow \mathcal F d^*X$ is $\mathcal S^\Delta$-fibrant.  Thus, we must show that $Y=\mathcal F d^*X$ satisfies conditions (2)--(5) of Definition \ref{cssobj} for a complete Segal object.

First we prove (2).  We have that
\[ Y([m],c) \simeq (d^*X)([m],c)= X([m](c,\dots,c)), \] while
\[ Y([1],c)\times_{Y([0],c)} \cdots \times_{Y([0],c)} Y([1],c) \cong X[1](c)\times_{X[0]}^h \cdots \times_{X[0]}^h X[1](c). \]
That (2) holds is then immediate from the fact that $X$ is $\Se^\Theta$-fibrant.

Next we prove (3).  We have that
\[ Y([m],t) \cong (d^*X)([m],t)= X[m](t,\dots,t) \]
for all $m$.  That is, there is a levelwise weak equivalence $\tau_\Theta^*X \rightarrow \tau_\Delta^*Y$ in $\sPsh(\Delta)$.
Then (3) follows from the fact that $X$ is $\Cpt^\Theta$-fibrant.

Next we establish (4).  For each object $c$ in $C$, the map $Y([1],c) \rightarrow Y([0],c)^{\times 2}$ (which is a fibration since $Y$ is injective fibrant and $\partial F[1] \times Fc \rightarrow F[1] \times Fc$ is an injective cofibration in $\sPsh(\Delta\times \mathcal C)$) is weakly equivalent to the map $X([1](c)) \rightarrow X([0])^{\times 2}$.  The latter map is also a fibration since $X$ is injective
fibrant and $F[0] \amalg F[0] \rightarrow F[1](c)$ is an injective cofibration in $\sPsh(\Theta \mathcal C)$), and
the map $X[0] \rightarrow Y([0],t)$ is a weak equivalence. Thus we have a levelwise weak equivalence
\[ M_X^\Theta(x_0,x_1)  = M^\Delta_{d_*X}(x_0, x_1)\rightarrow M_Y^\Delta(f(x_0),f(x_1)) \]
in $\sPsh(\mathcal C)$ for all $x_0, x_1 \in X[0]$.  Now (4) follows from the fact that $X$ is $\Rec^\Theta(\mathcal S)$-fibrant.

Lastly, we prove (5).  We have that $d^*X([0],c) = X([0])$, and thus $d^*X([0],t) \rightarrow d^*X([0],c)$ is an isomorphism for all $c$.  Therefore $Y([0],t) \rightarrow Y([0],c)$ is a weak equivalence for all $c$.
\end{proof}

\begin{theorem} \label{thetacompletesegal}
The adjoint pair
\[ d^*\colon \sPsh(\Theta \mathcal C)_{\mathcal S^\Theta} \rightleftarrows \sPsh(\Delta \times \mathcal C)_{\mathcal S^\Delta} \noloc d_* \]
is a Quillen equivalence.
\end{theorem}

\begin{proof}
Recall that all objects are cofibrant in both model categories.  We must prove that:
\begin{enumerate}
\item for any $\mathcal S^\Delta$-fibrant object $W$ in $\sPsh(\Delta \times \mathcal C)$, the map $d^*d_*W \rightarrow W$
is in $\overline{\mathcal S^\Delta}$, and

\item for any object $X$ in $\sPsh(\Theta \mathcal C)$, the composite
\[ X \rightarrow d_*d^*X \rightarrow d_* \mathcal F_{\mathcal S^\Delta}d^*X \]
is in $\overline{\mathcal S^\Theta}$, where $\mathcal F_{\mathcal S^\Delta}$ is a fibrant replacement functor in
$\sPsh(\Delta \times \mathcal C)_{\mathcal S^\Delta}$.
\end{enumerate}

First we prove (1).  By Propositions \ref{prop:d-lowerstar-fibrant} and \ref{prop:d-upperstar-local}, we know that $d^*d_*W$ is $\mathcal S^\Delta$-local.  Since the map in question is between two local objects, it suffices by Proposition \ref{prop:we-local-objects-delta} to show that the maps 
\[ (d^*d_*W)([0],t) \rightarrow W([0],t) \qquad \text{and} \qquad (d^*d_*W)([1],c) \rightarrow W([1],c) \]
are weak equivalences.  The first map is actually an isomorphism, using the fact that
\[ \tau_\Delta^*d^*d_* \cong \tau^*_\Theta d_*\cong (\pi_\Theta)_*d_* \cong (\pi_\Delta)_* \cong \tau_\Delta^*. \]
For the second map, note that
\[ (d^*d_*W)([1],c) \cong \Map_{\sPsh(\Delta\times \mathcal C)}(\Sigma(Fc),W), \]
and that, using the fact that $\Sigma(Fc)$ is defined via a pushout (see Section \ref{intertwiningsec}), the map can be identified in the top arrow of the pullback square
\[\xymatrix{ {\Map(\Sigma(Fc),W)} \ar[r] \ar[d]  & {\Map(F[1]\times Fc, W)} \ar[d] \\
{\Map(\partial F[1]\times Ft, W)} \ar[r] & {\Map(\partial F[1]\times Fc, W).}  }\]
The square is in fact a homotopy pullback square of spaces (since $W$ is injective fibrant and $\partial F[1]\times Fc \rightarrow F[1] \times Fc$ is an injective cofibration).  The result follows from the fact that the bottom horizontal arrow is a weak equivalence, since $W$ is $\mathcal S^\Delta$-fibrant.

Next we prove (2).  Choose an $\mathcal S^\Theta$-fibrant replacement $X \rightarrow \mathcal F_{\mathcal S^\Theta}X$, and consider the commutative square
\[\xymatrix{ X \ar[r] \ar[d]  & {d_* \mathcal F_{\mathcal S^\Delta}d^* X} \ar[d] \\
{\mathcal F_{\mathcal S^\Theta} X} \ar[r] & d_* \mathcal F_{\mathcal S^\Delta} d^*\mathcal F_{\mathcal S^\Theta} X }\]
in which the vertical arrows are in $\overline{\mathcal S^\Theta}$.  For the right-hand arrow, we use that
\[ d^*\colon \sPsh(\Theta \mathcal C)_{\mathcal S^\Theta} \rightleftarrows \sPsh(\Delta \times \mathcal  C)_{\mathcal S^\Delta}\noloc d_*\]
is a Quillen pair, by Corollary \ref{dquillenpair}, and that all objects are cofibrant. Thus, to show that the top horizontal arrow is in $\overline{\mathcal S^\Theta}$, it suffices to show the bottom horizontal arrow is.  Furthermore, since $d^*$ takes $\mathcal S^\Theta$-local objects to $\mathcal S^\Delta$-local objects by Proposition \ref{prop:d-upperstar-local}, we can replace the $\mathcal F_{\mathcal S^\Delta}$ in the lower right-hand corner with an injective fibrant replacement functor $\mathcal F$.  The resulting object $d_* \mathcal F d^*\mathcal F_{\mathcal S^\Theta} X$ is itself $\mathcal S^\Theta$-fibrant by Proposition \ref{prop:d-lowerstar-fibrant}.

Therefore, by Proposition \ref{prop:we-local-objects-theta} it suffices to prove that if $X$ is a $\mathcal S^\Theta$-fibrant object of $\sPsh(\Theta \mathcal C)$,  then the maps
\[ X[0] \rightarrow (d_*\mathcal F d^*X)[0],\qquad X[1](c) \rightarrow (d_*\mathcal F d^*X)[1](c) \]
are weak equivalences of spaces. The result follows using Lemma \ref{lemma:we-lemma-unit}.
\end{proof}

\begin{lemma} \label{lemma:we-lemma-unit}
Let $X$ be any object of $\sPsh(\Theta \mathcal C)$.  Then the maps
\[ X[0] \rightarrow (d_* \mathcal F d^*X)[0], \qquad X[1](c) \rightarrow (d_* \mathcal F d^*X)[1](c) \]
are weak equivalences of spaces.
\end{lemma}

\begin{proof}
In the first case, because $d^*(F[0]) \cong F[0] \times Ft$, the map is isomorphic to the weak equivalence
\[ (d^*X)([0],t) \rightarrow (\mathcal F d^*X)([0],t). \]

In the second case, because $d^*(F[1](c)) \cong \Sigma(Fc)$, the map is isomorphic to
\[ \Map(\Sigma(Fc), d^*X) \rightarrow \Map(\Sigma(Fc), \mathcal F d^*X). \]
Using the pushout decomposition
\[ \Sigma(Fc) \cong (F[1] \times Fc) \cup_{\partial F[1] \times Fc} \partial F[1] \times Ft, \]
we see that we can obtain this map by taking pullbacks of the rows in the diagram
\[\xymatrix{ {(d^*X)([1],c)} \ar[r] \ar[d]_{\simeq} & {(d^*X)([0],c)^{\times 2}} \ar[d]_{\simeq} & {(d^*X)([0],t)^{\times 2}} \ar[l]_-{g}^{\simeq} \ar[d]_{\simeq} \\
{(\mathcal F d^*X)([1],c)} \ar[r]^-{f} & {(\mathcal F d^*X)([0],c)^{\times 2}} & {(\mathcal F d^*X)([0],t)^{\times 2}} \ar[l] }\]
The vertical maps are weak equivalences of spaces.  The map marked $f$ is a fibration, since $\partial F[1] \times Fc \rightarrow F[1] \times Fc$ is an injective cofibration and $\mathcal F d^*X$ is injective fibrant. The map marked $g$ is an isomorphism (and thus a fibration), since $d_\#(F[0] \times Fc \rightarrow F[0] \times Ft)$ is an isomorphism.  Therefore, taking limits along rows computes homotopy pullbacks, and the induced map between them is a weak equivalence.
\end{proof}

\section{Equivalence between $\Thetansp$ and complete Segal objects in $\Theta_{n-1}Sp$}

In this section we look at consequences of the results of the previous section for models of $(\infty, n)$-categories.

Let $\Thetansp$ denote the model category $\sPsh(\Theta_{n})_{\mathcal S^\Theta}$, and let $\css(\Theta_{n-1}Sp)$ denote the model category $\sPsh(\Delta \times \Theta_{n-1})_{\mathcal S^\Delta}$.  Then we have the following special case of Theorem \ref{thetacompletesegal}.

\begin{cor}  \label{cssthetan}
The adjoint pair $(d^*, d_*)$ induces a Quillen equivalence of localized model categories
\[  d_* \colon \css(\Theta_{n-1}Sp) \leftrightarrows \Thetansp \colon d^*. \]
\end{cor}

However, this Quillen equivalence can be extended to a chain, by iterating the application of the adjoints $(d_*, d^*)$.

Observe that the functor $d$ can be iterated to obtain a chain of functors connecting $\Delta^n$ and $\Theta_n$:
\[ \Delta^n \rightarrow \Delta^{n-1} \times \Delta \rightarrow \Delta^{n-2} \times \Theta_2 \rightarrow \cdots \rightarrow \Delta \times \Theta_{n-1} \rightarrow \Theta_n. \]  This chain induces a string of Quillen pairs
\[ \sPsh(\Delta^n) \leftrightarrows \sPsh(\Delta^{n-1} \times \Delta) \leftrightarrows \cdots \leftrightarrows \sPsh(\Theta_n) \]
on the level of injective model structures.

Then applying Proposition \ref{almostcomplete} to each of these model categories, we have the following result.

\begin{cor}
For any $1 \leq i \leq n$, there exists a model structure $\css^i(\Theta_{n-i}Sp)$ on the category $\sPsh(\Delta^i \times \Theta_{n-i})$ in which the fibrant objects $W$ satisfy the following conditions:
\begin{enumerate}
\item $W$ is injective fibrant,

\item $W(-, x)$ is a Segal space for each object $x$ in $\Delta^{i-1} \times \Theta_{n-i}$,

\item $W([m], -)$ is a complete Segal object in $\css^{i-1}(\Theta_{n-i}Sp)$ for each object $[m]$ of $\Delta$,

\item $W(-, t)$ is a complete Segal space, where $t$ is the terminal object of $\Delta^{i-1} \times \Theta_{n-i}$, and

\item $W([0], -)$ is essentially constant..
\end{enumerate}
\end{cor}

Observe that in the case where $i=0$, if we take $\css^0(\Theta_{n-1}Sp)$ to be $\Theta_{n-1}Sp$, then we obtain complete Segal objects in $\Theta_{n-1}Sp$, as described in Corollary \ref{cssthetan}.  Furthermore, if $i=n$, this description of the fibrant objects exactly coincides with the Barwick-Lurie definition of $n$-fold complete Segal spaces \cite[2.1.37]{luriecob}.  Applying Theorem \ref{thetacompletesegal} to the adjoint pairs described above, we obtain the following result.

\begin{cor}
There is a chain of Quillen equivalences
\[ \css^n(\SSets) \leftrightarrows \css^{n-1}(\Theta_1Sp) \leftrightarrows \cdots \leftrightarrows \css(\Theta_{n-1}Sp) \leftrightarrows \Thetansp. \]
\end{cor}

\section{Notions of equivalence in complete Segal objects}

The goal for the remainder of the paper is to establish the final desired Quillen equivalence of model categories, namely the one between Segal category objects and complete Segal objects.  Ultimately, we must restrict to complete Segal objects in $\sPsh(\Delta \times \Theta_n)$, rather than in more general $\sPsh(\Delta \times \mathcal C)$, since currently we only have a model structure for Segal category objects in this context.  We refer to complete Segal objects in $\sPsh(\Delta \times \Theta_n)$ as complete Segal objects in $\Thetansp$ and denote the model structure by $\css(\Thetansp)$.  Most of the constructions to this section that we make in this setting can be done in the more general context, however.

We also make use of the model structure $\sPsh(\Delta \times \Theta_n)$ whose fibrant objects satisfy conditions \eqref{injfib}-\eqref{complete0} of Definition \ref{cssobj}, and hence condition (4') by Proposition \ref{4implies4'}, but not necessarily condition \eqref{essconst}.  Recall that we refer to these fibrant objects as \emph{almost complete Segal objects} and denote the corresponding model structure by $\acss(\Thetansp)$.

It is often convenient here to think of (complete) Segal objects here as functors $\Deltaop \rightarrow \Thetansp$, although at times we continue to think of them as functors $\Deltaop \times \Thetanop \rightarrow \SSets$ as previously.  In the latter perspective, when specifying the terminal objects of $\Delta$ and $\Theta_n$, we sometimes clarify by denoting them by $[0]_\Delta$ and $[0]_\Theta$, respectively.

\subsection{Segal objects in $\Thetansp$}

Given a Segal object $W$, define its \emph{objects} to be the elements of the set $\ob(W) = W([0]_\Delta, [0]_\Theta)_0$.   Given two elements $x,y \in \ob(W)$, we define the \emph{mapping object} $M_W^\Delta(x,y)$ as the fiber of the map $(d_1, d_0) \colon W_1 \rightarrow W_0 \times W_0$ over the point $(x,y)$, which is a presheaf on $\Theta_n$.  Since $W$ is assumed to be injective fibrant, which is precisely Reedy fibrant in this case, the mapping object is homotopy invariant.  Observe that this definition coincides with the one given in Section \ref{mapping}.

As we have seen, we can obtain a simplicial space from a $\Theta_n$-space via the functor $\tau_\Theta^\ast$.  Taking the space in degree zero of $\tau_\Theta^*M_W^\Delta(x,y)$, we get mapping spaces for $W$.  As a consequence, it is possible to define the \emph{homotopy category} $\Ho(W)$ of a Segal object $W$ as the homotopy category of the Segal space $\tau^*W$, where the objects are those of $W$ and the morphisms are the components of the mapping spaces.

\setcounter{theorem}{1}

\begin{definition}
A map $f \colon W \rightarrow Z$ of Segal objects in $\Thetansp$ is a \emph{Dwyer-Kan equivalence} if:
\begin{itemize}
\item for any objects $x$ and $y$ of $W$, the induced map $M_W^\Delta(x,y) \rightarrow M_Z^\Delta(fx,fy)$ is a weak equivalence in $\Thetansp$, and

\item the induced functor $\Ho(W) \rightarrow \Ho(Z)$ is an equivalence of categories.
\end{itemize}
\end{definition}

In order to compare complete Segal space objects to Segal category objects, we need to understand the relationship between the weak equivalences in the two model structures.  The purpose of this section is to establish that weak equivalences between Segal objects in $\css(\Thetansp)$ are precisely Dwyer-Kan equivalences.

For any $x \in \ob(W)= W([0]_\Delta, [0]_\Theta)_0$, define its \emph{identity map} $\id_X$ to be $s_0(x) \in M_W^\Delta(x,x)_0$.  Two maps $f, g \in M_W^\Delta(x,y)_0$ are \emph{homotopic} if they lie in the same component of the induced mapping space, denoted by $f \simeq g$.

Recall the more general mapping object $M_W^\Delta(x_0, \ldots, x_k)$ from Section \ref{mapping}; it is a functor $\Thetanop \rightarrow \SSets$ and comes equipped with
induced acyclic fibrations
\[ \varphi_k \colon M_W^\Delta(x_0, \ldots, x_k) \rightarrow M_W^\Delta(x_0, x_1) \times \cdots \times M_W^\Delta(x_{k-1}, x_k). \]

We can then define composition in a Segal object $W$.  Given $(f,g) \in M_W^\Delta(x,y)_0 \times M_W^\Delta(y,z)_0$, a \emph{composition} is a lift of $\varphi_2$ to some $k \in M_W^\Delta(x,y,z)_0$; a \emph{result} of this composition is $d_1(k) \in M_W^\Delta(x,z)_0$.  While $k$ is not uniquely determined, since $\varphi_2$ is an acyclic fibration we can conclude that any two choices for $k$ give results that are homotopic.  Therefore, we write $g \circ f$ for the result of some composition of $f$ and $g$.

The following proposition can be proved as for ordinary Segal spaces \cite[5.4]{rezk}.

\begin{prop}
Let $W$ be a Segal object.  Given $f \in M_W^\Delta(w,x)_0$, $g \in M_W^\Delta(x,y)_0$, and $h \in M_W^\Delta(y,z)_0$, we have $(h \circ g) \circ f \simeq h \circ (g \circ f)$ and $f \circ \id_w \simeq \id_x \circ f$.
\end{prop}

\begin{definition}
Let $W$ be a Segal object.  An element $g \in M_W^\Delta(x,y)_0$ is a \emph{homotopy equivalence} if there exist $f, h \in M_W^\Delta(y,x)_0$ such that $g \circ f \simeq \id_y$ and $h \circ g \simeq \id_x$.
\end{definition}

Observe in this case that $h \simeq h \circ g \circ f \simeq f$ and that, for any $x \in \ob(W)$, $\id_x \in M_W^\Delta(x,x)_0$ is a homotopy equivalence.  Once again, the following lemma can be proved just as for complete Segal spaces \cite[5.8]{rezk}.

\begin{prop}
If $g \in W([1],[0]_\Theta)_0$ and there exists a path $F_\Theta[1]([0]) \rightarrow W_1$ from $g$ to a homotopy equivalence $g' \in W([1], [0]_\Theta)_0$, then $g$ is a homotopy equivalence.
\end{prop}

\begin{definition}
Define the \emph{homotopy equivalences} of $W$, denoted by $W_{\heq}$, to consist of those components of $W_1$ containing homotopy equivalences.
\end{definition}

Now observe that the degeneracy map $s_0 \colon W_0 \rightarrow W_1$ factors through $W_{\heq}$.   We would like a complete Segal object in $\Thetansp$ to satisfy the condition that the map $s_0 \colon W_0 \rightarrow W_{heq}$ is a weak equivalence in $\Thetansp$.  As in the case of ordinary complete Segal spaces, we can show that this condition is equivalent to condition (4'). 

\setcounter{subsection}{6}

\subsection{Categorical equivalences}

\setcounter{theorem}{7}

Recall that $E$ is the nerve of the category with two objects and a single isomorphism between them.  We can treat it as a discrete simplicial object in $\Thetansp$.  

\begin{definition}
Let $f, g \colon U \rightarrow V$ be maps between Segal space objects.  A \emph{categorical homotopy} between $f$ and $g$ is given by a map $H \colon U \times E \rightarrow V$ such that the diagram
\[ \xymatrix{U \ar[d]_{\id \times i_0} \ar[dr]^f & \\
U \times E \ar[r]^H & V \\
U \ar[u]^{\id \times i_1} \ar[ur]_g & } \]
commutes.  Equivalently, a categorical homotopy is given by a map $H'$ or $H''$ as given in its respective commutative diagram
\[ \xymatrix{& V & F \ar[d]_{i_0} \ar[dr]^{\{f\}} & \\
U \ar[r]^{H'} \ar[ur]^f \ar[dr]_g & V^E \ar[u]_{V^{i_0}} \ar[d] ^{V^{i_1}} & E \ar[r]^{H''} & V^U \\
&  V & F[0] \ar[u]^{i_1} \ar[ur]_{\{g\}} & . } \]
\end{definition}

\begin{prop} \label{catsimpequiv}
Suppose that $U$ and $W$ are Segal objects and $W$ is almost complete.  Then maps $f, g \colon U \rightarrow W$ are categorically homotopic if and only if there exists a homotopy $K \colon U \rightarrow W^{\Delta[1]}$ which restricts to $f$ and $g$ on its endpoints.
\end{prop}

\begin{proof}
Consider the following diagram in the category of functors $\Deltaop \rightarrow \Thetansp$ over $W \times W$:
\[ \xymatrix{W \ar[r] \ar[d] & W^E \ar[d] \\
W^{F[1]} \ar[r] & W \times W} \]
where the maps out of $W$ are given by the inclusion of an object into either $F[1]$ or $E$, as appropriate.  The former is a levelwise weak equivalence since it is the inclusion of constant paths in the path fibration.  The latter is a levelwise weak equivalence because $W$ satisfies condition (4').  We know further that the maps to $W \times W$ are both fibrations in the Reedy model structure.

If $f$ and $g$ are simplicially homotopic, then we have a map $K \colon U \rightarrow W^{F[1]}$ such that the composite with the map to $W \times W$ is precisely $(f,g)$.  To obtain a categorical homotopy, then we need this composite map to factor through $W^E$.  If the maps out of $W$ were fibrations (in particular if $W$ were a pullback), then we could factor $K$ though $W$.  However, they are not.

To remedy this situation, take the pullback
\[ P = W^{F[1]} \times_{W \times W} W^E \]
and then factor the natural map $W \rightarrow P$ as a levelwise acyclic cofibration followed by a fibration $W \rightarrow Q \rightarrow P$.  Now the maps $W^{F[1]} \leftarrow P \rightarrow W^E$ are fibrations, since $P$ is a pullback along fibrations, and by composition so are the maps $W^{F[1]} \leftarrow Q \rightarrow W^E$.  But we also know that these two maps are levelwise weak equivalences, by the 2-out-of-3 property.  Therefore, a lift
\[ \xymatrix{& Q \ar@{->>}[d]^\simeq \\
U \ar[r] \ar@{-->}[ur] & W^{F[1]}} \]
exists, which we can compose with $Q \rightarrow W^E$ to obtain the desired categorical homotopy.

Now observe that this argument is completely symmetric; if we began instead with a categorical homotopy $U \rightarrow W^E$, we can produce a simplicial homotopy $U \rightarrow W^{F[1]}$.
\end{proof}

\begin{definition}
A map $g \colon U \rightarrow V$ of Segal space objects is a \emph{categorical equivalence} if there exist maps $f, h \colon V \rightarrow U$ together with categorical homotopies $gf \simeq \id_V$ and $fg \simeq \id_U$.
\end{definition}

\begin{prop} \label{catreedy}
A map $g \colon U \rightarrow V$ between almost complete Segal objects is a categorical equivalence if and only if it is a levelwise weak equivalence.
\end{prop}

\begin{proof}
We know from Proposition \ref{catsimpequiv} that $g$ is a categorical equivalence if and only if it is a simplicial homotopy equivalence.  Since $U$ and $V$ Reedy fibrant and cofibrant, then simplicial homotopy equivalence coincides with Reedy weak equivalence.
\end{proof}

We now show that categorical equivalences are compatible with the cartesian structure.

\begin{prop}
Let $U$, $V$, and $W$ be Segal objects.  If $f, g \colon U \rightarrow V$ are categorically homotopic maps, then so are the induced maps $W^f, W^g \colon W^V \rightarrow W^U$.
\end{prop}

\begin{proof}
Let $H \colon U \times E \rightarrow V$ be a categorical homotopy between $f$ and $g$.  Then
\[ W^H \colon W^V \rightarrow W^{U \times E} \cong (W^U)^E \]
defines a categorical homotopy between $W^f$ and $W^g$.
\end{proof}

\begin{cor}  \label{internal}
If $U \rightarrow V$ is a categorical equivalence between Segal objects, then so is $W^V \rightarrow W^U$.
\end{cor}

Recall that in our current setting the model structure $\acss$ agrees with the model structure of Theorem \ref{123'}, which is cartesian.

\begin{prop}
A categorical equivalence $f \colon U \rightarrow V$ between Segal objects is a weak equivalence in $\acss$.
\end{prop}

\begin{proof}
We want to show that $\Map(f, W)$ is a weak equivalence of simplicial sets for any Segal object $W$ satisfying (4').  Consider $W^f \colon W^V \rightarrow W^U$.  Since $W$ is a fibrant object in a cartesian model structure, so are both $W^V$ and $W^U$.  Therefore, $W^f$ is a categorical equivalence by Corollary \ref{internal} and therefore a levelwise weak equivalence by Proposition \ref{catreedy}.  But then passing to the mapping space $\Map(f, W)$ is still a weak equivalence, proving the proposition.
\end{proof}

\begin{cor} \label{catcss}
A categorical equivalence between Segal objects is a weak equivalence in $\acss(\Thetansp)$.
\end{cor}

\setcounter{subsection}{15}

\subsection{Dwyer-Kan equivalences}

The following result is the analogue of \cite[7.6]{rezk}.  Observe that so far in this section we have not used condition \eqref{essconst} for complete Segal space objects.  However, when we consider Dwyer-Kan equivalences, it is necessary to use this property.

\setcounter{theorem}{16}

\begin{prop} \label{dkreedy}
A map $f \colon U \rightarrow V$ between complete Segal objects is a Dwyer-Kan equivalence if and only if it is a levelwise weak equivalence.
\end{prop}

\begin{proof}
A levelwise weak equivalence is always a Dwyer-Kan equivalence, so we need only prove the reverse implication.  Suppose that $U$ and $V$ are complete Segal objects and that $f \colon U \rightarrow V$ is a Dwyer-Kan equivalence.  Then for any $x,y \in U([0]_\Delta,[0]_\Theta)$, we have that $M_U^\Delta(x,y) \rightarrow M_V^\Delta(fx,fy)$ is a weak equivalence in $\Thetansp$, and that $\Ho(U) \rightarrow \Ho(V)$ is an equivalence of categories.

Recall that $M_U^\Delta(x,y)$ can be written as a (homotopy) pullback
\[ \xymatrix{M_U^{\Delta}(x,y) \ar[r] \ar[d] & U_1 \ar[d] \\
\{(x,y)\} \ar[r] & U_0 \times U_0} \]
and that $U_{\heq}$ consists of components of $U_1$.  Therefore, we can define the mapping object of equivalences $M^\Delta_U(x,y)_{\heq}$ by restricting $U_1$ to $U_{\heq}$ and taking the pullback
\[ \xymatrix{M^\Delta_U(x,y)_{\heq} \ar[r] \ar[d] & U_{\heq} \ar[d] \\
\{(x,y)\} \ar[r] & U_0 \times U_0.} \]
We obtain that for any pair of objects $(x,y)$, $M^\Delta_U(x,y)_{\heq} \rightarrow M_V^\Delta(fx,fy)_{\heq}$ is a weak equivalence in $\Thetansp$.

By precomposing with the degeneracy map $s_0 \colon U_0 \rightarrow U_{\heq}$, which is an equivalence, we obtain a commutative diagram
\[ \xymatrix{U_0 \ar[r] \ar[d] & U_0 \times U_0 \ar[d] \\
V_0 \ar[r] & V_0 \times V_0} \]
which is a homotopy pullback diagram since taking the horizontal fibers results in a weak equivalence $M^\Delta_U(x,y)_{\heq} \rightarrow M^\Delta_V(fx,fy)_{\heq}$.  Because we have assumed that $U$ and $V$ are complete Segal objects, so that $U_0$ and $V_0$ are essentially constant when evaluating at objects of $\Theta_n$, it follows that the map $U_0 \rightarrow V_0$ is a weak equivalence.  Then $U_1 \rightarrow V_1$ must also be a weak equivalence, since
\[ \xymatrix{U_1 \ar[r] \ar[d] & U_0 \times U_0 \ar[d] \\
V_1 \ar[r] & V_0 \times V_0} \]
is a pullback diagram with horizontal maps fibrations, since $U$ and $V$ are Reedy fibrant, and therefore a homotopy pullback diagram.  Using the Segal condition, we have established that $f$ is a Reedy weak equivalence.
\end{proof}

However, we need the following stronger result.

\begin{theorem}  \label{dkcss}
A map $f \colon U \rightarrow V$ of Segal objects is a Dwyer-Kan equivalence if and only if it is a weak equivalence in the model category $\css(\Thetansp)$.
\end{theorem}

The heart of the proof is in the following method for completing a Segal object.

\begin{lemma} \label{completion}
Given any Segal object $W$, there exists a completion map $i_W \colon W \rightarrow \widehat W^d$ such that:
\begin{enumerate}
\item the completion $\widehat W^d$ is a complete Segal object;

\item the map $i_W$ is a weak equivalence in $\css(\Thetansp)$; and

\item the map $i_W$ is a Dwyer-Kan equivalence.
\end{enumerate}
\end{lemma}

\begin{proof}
Generalizing the discrete object $E$ above, let $E(k)$ denote the nerve of the category with $(k+1)$ objects and a single isomorphism between any two of them, regarded as a discrete simplicial object in $\Thetanop$.

Since $W$ is a Segal object, if we fix an object of $\Thetanop$, we have a Segal space in the ordinary sense to which we can apply completion.  Therefore, we simply define the completion of $W$ to be given by the levelwise completion of Segal spaces to complete Segal spaces as defined in \cite[\S 14]{rezk}.   More precisely, define the completion
\[ \begin{aligned}
\widehat W & = \diag_\Delta \left([m] \mapsto W^{E(m)} \right) \\
& = \hocolim_\Delta \left( [m] \mapsto W^{E(m)} \right).
\end{aligned}\]

Now observe that $W$ can be regarded as a homotopy colimit of the constant diagram given by $W$ itself.  We know that the unique map $E(m) \rightarrow \Delta[0]$ induces a categorical equivalence $W \rightarrow W^{E(m)}$ which, by Proposition \ref{catcss} is a weak equivalence in the model structure $\acss(\Thetansp)$, hence also in $\css(\Thetansp)$.  The map $W \rightarrow \widehat W$ is given by the induced map on homotopy colimits and is hence a weak equivalence.   The fact that $\widehat{W}$ is a fibrant object in $\acss(\Thetansp)$ follows from the fact that we defined it as a levelwise completion, so it satisfies the required conditions, i.e., \eqref{injfib}-\eqref{complete0} Definition \ref{cssobj}, and hence also (4') from Remark \ref{4vs4'}.

Next, we show that the map $W \rightarrow \widehat W$ is a Dwyer-Kan equivalence.  By construction, it is essentially surjective, in that it induces an essentially surjective map on homotopy categories, so we need only check the condition on mapping objects.  Consider the 0-coskeleton $\cosk_0(W_0)$ which has as $k$-simplices the object $W_0^{k+1}$. We can apply Lemma \ref{hpb} to the map $W \rightarrow \cosk_0(W_0)$ and the homotopy pullback diagrams
\[ \xymatrix{W_q \ar[r] \ar[d] & W_p \ar[d] \\
W_0^{q+1} \ar[r] & W_0^{p+1}} \]
ranging over all $[p] \rightarrow [q]$ in $\Delta$ to conclude that
\[ \xymatrix{W_k \ar[r] \ar[d] & \widehat W_k \ar[d] \\
W_0^{k+1} \ar[r] & \widehat W_0^{k+1}} \]
is a homotopy pullback square.  The homotopy fibers of the vertical maps are precisely the mapping objects of $W$ and $\widehat W$, respectively, which are hence weakly equivalent.  The case where $k=2$ gives the required condition for a Dwyer-Kan equivalence.

We now modify $\widehat W$ so that it satisfies condition \eqref{essconst}.  Consider the $\Theta_n$-space $\widehat W_0$, and the simplicial object $\cosk_0(\widehat W_0)$, equipped with a natural map $\widehat W \rightarrow \cosk_0(\widehat W_0)$.  Similarly, think of $\widehat W([0]_\Delta, [0]_\Theta)$ as a constant $\Theta_n$-space and take its coskeleton $\cosk_0(\widehat W([0]_\Delta, [0]_\Theta))$.  Define $\widehat W^d$ to be the pullback of the diagram
\[ \xymatrix{\widehat W^d \ar[r] \ar[d] & \cosk_0(\widehat W([0]_\Delta, [0]_\Theta)) \ar[d] \\
\widehat W \ar[r] & \cosk_0(\widehat W_0).} \]
Then $\widehat{W}^d$ is still injective fibrant, and it still satisfies condition \eqref{segal}.  Looking at degree zero, we obtain that $\widehat{W}^d_0 =  \widehat{W}([0]_\Delta,[0]_\Theta)$.  In higher degrees, we get that
\[ \widehat{W}^d_n \simeq \coprod_{x,y \in \widehat{W}([0]_\Delta ,[0]_\Theta)} M^\Delta_{\widehat{W}}(x,y). \]
In particular, 
\[ M^\Delta_{\widehat{W}^d}(x,y) = M^\Delta_{\widehat{W}}(x,y) \] 
for all $x,y \in \widehat{W}([0]_\Delta,[0]_\Theta)$, so $\widehat{W}^d$ satisfies \eqref{localmapping}.  Then one can check that condition \eqref{complete0} is also unchanged by taking this construction.    By design, $\widehat{W}^d$ additionally satisfies condition \eqref{essconst}, and it comes with a natural completion map $i_W \colon W \rightarrow \widehat W^d$, given by the composite $W \rightarrow \widehat{W} \rightarrow \widehat{W}^d$.  

Since we have already seen that the first map $W \rightarrow \widehat{W}$ is a Dwyer-Kan equivalence, it suffices to prove that the second map $\widehat{W} \rightarrow \widehat{W}^d$ is also.  We have just shown that this map is the identity on mapping objects, and since it is also the identity on objects, it must be a Dwyer-Kan equivalence.  

Finally, we need to know that $i_W$ is a weak equivalence in $\css(\Thetansp)$.  We know already that $W \rightarrow \widehat{W}$ is a weak equivalence in $\acss(\Thetansp)$, hence also in $\css(\Thetansp)$, so it suffices to check the map $\widehat{W} \rightarrow \widehat{W}^d$.  Let $Z$ be a complete Segal object in $\Thetansp$, and consider the map induced map $\Map(\widehat{W}^d, Z) \rightarrow \Map(\widehat{W}, Z)$.  However, since $Z$ itself satisfies condition \eqref{essconst}, any map $\widetilde{W} \rightarrow Z$ is determined by a map $\widehat{W}^d \rightarrow Z$, so the map on mapping spaces is a weak equivalence, as we needed to show.
\end{proof}

The following lemma, which we have used in the above proof, is essentially given by a descent argument.

\begin{lemma} \label{hpb}
Let $X \rightarrow Y$ be a map of simplicial objects in $\Thetansp$ such that, for every map $[p] \rightarrow [q]$ in $\Delta$, the induced diagram
\[ \xymatrix{X_q \ar[r] \ar[d] & X_p \ar[d] \\
Y_q \ar[r] & Y_p} \]
is a levelwise homotopy pullback square.  Then the diagram
\[ \xymatrix{X_0 \ar[r] \ar[d] & \hocolim_\Delta X \ar[d] \\
Y_0 \ar[r] & \hocolim_\Delta Y} \]
is also a levelwise homotopy pullback square.
\end{lemma}

We are now able to prove the main result of this section.

\begin{proof}[Proof of Theorem \ref{dkcss}]
Let $f \colon U \rightarrow V$ be a map of Segal objects.  Using the completion functor from Lemma \ref{completion}, we have that $f$ is a weak equivalence in $\css(\Thetansp)$ if and only if $\widehat f \colon \widehat{U}^d \rightarrow \widehat{V}^d$ is a weak equivalence.   But $\widehat f$ is a map of complete Segal space objects, which are the fibrant objects in a localized model structure, so $\widehat f$ is a weak equivalence if and only if it is a levelwise weak equivalence.  But then by Proposition \ref{dkreedy}, $\widehat f$ is a levelwise weak equivalence if and only if $\widehat f$ is a Dwyer-Kan equivalence.  Finally, since Dwyer-Kan equivalences satisfy the 2-out-of-3 property we also know that $f$ is a Dwyer-Kan equivalence if and only if $\widehat f$ is.   Therefore, $f$ is a weak equivalence if and only if it is a Dwyer-Kan equivalence.
\end{proof}

\section{Equivalence between Segal category and complete Segal space objects in $\Thetansp$}

With the results of the last section, we are now able to give the Quillen equivalence between the model structures for Segal category objects and complete Segal objects.  We begin by recalling the model structure for Segal category objects from \cite{part1}.

\begin{definition}
A \emph{Segal category object in} $\Thetansp$ is a Segal object $X$ in $\Thetansp$ such that $X_0$ is discrete.
\end{definition}

\begin{theorem} \cite[6.9]{part1}
There is a cofibrantly generated model structure $\Secat(\Thetansp)$ on the category of functors $X \colon \Deltaop \rightarrow \Thetansp$ with $X_0$ discrete, with structure defined as follows.
\begin{itemize}
\item Weak equivalences are the maps $f \colon X \rightarrow Y$ such that the induced map $LX \rightarrow LY$ is a Dwyer-Kan equivalence of Segal space objects.

\item Cofibrations are the monomorphisms.

\item Fibrations are the maps with the right lifting property with respect to the maps which are both cofibrations and weak equivalences.

\item Fibrant objects are precisely the Segal category objects.
\end{itemize}
\end{theorem}

Recall that the underlying objects of the model category for complete Segal space objects in $\Thetansp$ are functors $\Deltaop \rightarrow \Thetansp$, whereas the underlying objects of the model category for Segal category objects are those functors $\Deltaop \rightarrow \Thetansp$ which take the object $[0]$ of $\Deltaop$ to a discrete object in $\Thetansp$, i.e., just a set.  Therefore, we have an inclusion functor
\[ I \colon \Secat(\Thetansp) \rightarrow \css(\Thetansp). \]   We aim to describe a right adjoint functor to $I$ and show that this adjoint pair defines a Quillen equivalence of model categories.

Let $W$ be an object of $\css(\Thetansp)$, thought of as a functor $\Deltaop \rightarrow \Thetansp$.   Consider the functor $U = \cosk_0(W_0)$.  Thinking instead of $W$ as a functor $\Deltaop \times \Thetanop \times \Deltaop \rightarrow \Sets$, define $V=\cosk_0(W([0]_\Delta, [0]_\Theta)_0)$.

Using the natural map $W \rightarrow U$ and the inclusion $V \rightarrow W$, define $RW$ to be the pullback
\[ \xymatrix{RW \ar[r] \ar[d] & V \ar[d] \\
W \ar[r] & U. } \]
Since $W$ is assumed to be Reedy fibrant, the map $W \rightarrow U$ is a fibration and therefore $RW$ is actually a homotopy pullback.  The diagrams $U$ and $V$ are still Segal spaces (and in particular Reedy fibrant) but in general they are no longer complete.

Observe that $RW$ has discrete $\Theta_n$-space in degree 0.  Therefore, we have defined a functor
\[ R \colon \css(\Thetansp) \rightarrow \Secat(\Thetansp). \]

\begin{lemma}
If $W$ is a complete Segal object, then $RW$ is a Segal category object.
\end{lemma}

\begin{proof}
By construction, we know that $RW$ has the necessary discreteness property in degree 0 since $V$ does.   As we observed above, both $U$ and $V$ are both Segal objects, and so in particular $V$ is a Segal category object.  Using the definition of coskeleton, we get a diagram of $\Theta_n$-spaces in degree 1
\[ \xymatrix{(RW)_1 \ar[r] \ar[d] & W([0]_\Delta, [0]_\Theta)_0 \times W([0]_\Delta, [0]_\Theta)_0  \ar[d] \\
W_1 \ar[r] & W_0 \times W_0.}  \]
Looking at the corresponding pullback diagram in degree 2,
\[ \xymatrix{(RW)_2 \ar[r] \ar[d] & (W_{0, [0], 0})^3 \ar[d] \\
W_2 \simeq W_1 \times_{W_0} W_1 \ar[r] & (W_0)^3}  \]
we can see that $(RW)_2$ must be weakly equivalent to $(RW)_1 \times_{(RW)_0} (RW)_1$.  Proceeding similarly in higher degrees, we have that $RW$ satisfies the Segal condition.
\end{proof}

\begin{prop}
The functor $R$ is right adjoint to the inclusion functor $I$.
\end{prop}

\begin{proof}
We need to prove that, for any object $Y$ in $\Secat(\Thetansp)$ and any object $W$ in $\css(\Thetansp)$,
\[ \Hom_{\Secat(\Thetansp)}(Y, RW) \cong \Hom_{\css(\Thetansp)}(IY, W). \]
Since $I$ is just the inclusion functor, note that $IY=Y$.

Suppose we have a map $Y \rightarrow W$.  Since $Y_0$ is assumed to be discrete, $Y_0 = Y([0]_\Delta, [0]_\Theta)_0$, viewed as a constant functor $\Thetanop \rightarrow \SSets$.  Therefore, if $V = \cosk_0(W([0]_\Delta,[0]_\Theta)_0)$ as before, the map $Y \rightarrow W$ factors uniquely through $V$.  The universal property of pullbacks then gives a unique map $Y \rightarrow RW$.  Thus, we have described the functor $\varphi \colon \Hom(Y, W) \rightarrow \Hom(Y, RW)$.  We claim that $\varphi$ is an isomorphism.

The map $\varphi$ is surjective, since any map $Y \rightarrow RW$ arises from the composite $Y \rightarrow RW \rightarrow W$.  Furthermore, it is injective because of the uniqueness constraints in the definition of $\varphi$.
\end{proof}

\begin{prop}
The adjoint pair $(I,R)$ defines a Quillen pair.
\end{prop}

\begin{proof}
First, observe that $I$ preserves cofibrations, since in each model category the cofibrations are precisely the monomorphisms.

It remains to show that $I$ preserves acyclic cofibrations.  First recall that a map in $\Secat(\Thetansp)$ is weak equivalence precisely if the induced map on fibrant replacements is a Dwyer-Kan equivalence.  But these fibrant replacements are Segal category objects, and therefore Segal space objects.  By Corollary \ref{dkcss}, this map must be a weak equivalence in $\css(\Thetansp)$.  Therefore, the inclusion functor $I$ preserves weak equivalences, in particular acyclic cofibrations.
\end{proof}

\begin{theorem}
The adjoint pair $(I, R)$ is a Quillen equivalence.
\end{theorem}

\begin{proof}
Using the argument in the previous proof, we can see that $I$ reflects weak equivalences, and in particular those between cofibrant objects.  Thus, it only remains to show that, given any fibrant object $W$ in $\css(\Thetansp)$, the map $IRW = RW \rightarrow W$ is a weak equivalence in $\css(\Thetansp)$.  Observe that this map is just the left-hand vertical map $j$ in the pullback diagram
\[ \xymatrix{RW \ar[r] \ar[d] & V \ar[d] \\
W \ar[r] & U.} \]
Since both $W$ and $RW$ are Segal space objects, it suffices to verify that $j$ is a Dwyer-Kan equivalence.  By the construction of $RW$, observe that $\ob(W)=\ob(RW)$, and $j$ is the identity map on object sets.  For any $x,y \in \ob(W)$, consider $M_{RW}^\Delta(x,y) \rightarrow M_W^\Delta(x,y)$.  Using the pullback describing $(RW)_1$ and restricting to this mapping object, we see that this map must be a weak equivalence in $\Thetansp$.
\end{proof}

\end{document}